\documentclass[10pt]{article}

\usepackage[colorlinks,citecolor=blue,urlcolor=blue]{hyperref}

\usepackage{amssymb}
\usepackage{amsmath}
\usepackage{amsthm}
\usepackage{srcltx}
\usepackage[authoryear]{natbib}
\usepackage{mathrsfs}
\usepackage{graphicx}

\long\def\beginskip#1\endskip{}
\def\endskip{}
\newcommand\bY{\mathbf{Y}}

\newcommand{\PP}{{\rm P}}
\newcommand{\EE}{{\rm E}}

\newcommand{\lle}{\,\,{\lesssim}\,\,}

\newcommand{\Var}{\mathrm{Var}}

\newcommand{\eps}{\varepsilon}

\renewcommand{\phi}{\varphi}

\newcommand{\given}{\,|\,}

\renewcommand{\a}{\alpha}
\newcommand{\argmax}{\mathop{\rm argmax}}
\newcommand{\argmin}{\mathop{\rm argmin}}

\renewcommand{\b}{\beta}
\renewcommand{\t}{\tau}

\newcommand{\tr}{{\rm tr}}

\newcommand\blfootnote[1]{%
  \begingroup
  \renewcommand\thefootnote{}\footnote{#1}%
  \addtocounter{footnote}{-1}%
  \endgroup
}

\newtheorem{theorem}{Theorem}[section]
\newtheorem{lemma}[theorem]{Lemma}

\theoremstyle{definition}

\begin{document}


\title{\bf An asymptotic analysis of distributed nonparametric methods\\[.5cm]}

\author{
Botond Szab\'o\\[.5ex]
{\em Leiden University}
\and 
Harry van Zanten\\[.5ex]
{\em University of Amsterdam}
\blfootnote{The research leading to these results has received funding from 
the Netherlands Science foundation NWO and from the European Research Council 
under ERC Grant Agreement 320637.}}

\date{}


\maketitle

%

\begin{abstract}
We investigate and compare the fundamental performance of several distributed 
learning methods that have 
been proposed recently. 
We do this in the context of a  distributed version of the classical 
signal-in-Gaussian-white-noise model, which  serves as a benchmark model for 
studying performance in this setting.
The results show how the design and tuning of a distributed method
 can have great impact on convergence rates and validity of uncertainty quantification.
  Moreover, we highlight the difficulty of  designing 
 nonparametric distributed procedures that automatically adapt to smoothness.
\end{abstract}

%
%


\numberwithin{equation}{section}

\section{Introduction}

Both in  statistics and machine learning there has been substantial interest in the design and 
study of distributed statistical or learning methods in recent years. One driving reason is the fact that in 
certain applications datasets have become so large that it is often unfeasible, or computationally undesirable, 
to carry out the analysis on a single machine. In a  distributed method the data are divided over 
a cluster consisting of several machines and/or cores. The  machines in the cluster  then process their   
data locally, after which 
the local results are somehow aggregated on a central machine to finally produce the overall outcome of 
the statistical analysis. Distributed methods are not only used for computational reasons, but are 
for instance also  of interest in situations where privacy 
is important and it is undesirable that all data are handled at a single location. Moreover, 
there are applications in which data are by construction gathered at multiple locations and 
first processed locally, before being combined at a central location.

Over the last years a  variety of distributed methods have been proposed. 
Recent examples include Consensus Monte Carlo (\cite{scott:etal:2013}), WASP (\cite{srivastava:2015a}), 
and distributed GP's (\cite{deisenroth:2015}), to mention but a few.
Most papers on distributed methods do extensive experiments on simulated, benchmark and real data to numerically assess
and compare the performance of the various methods. Some papers also 
derive a number of theoretical properties. 
Theoretical results on the performance 
of distributed methods are not yet widely available however and  there is certainly 
no common theoretical framework in place that allows a clear theoretical comparison of methods
and the development of an understanding of fundamental performance guarantees and limitations. 
 
Since a better  theoretical understanding of distributed methods can help to pinpoint 
fundamental difficulties and opportunities, we develop a  framework in this paper 
which allows us to study and compare the performance of various methods. We are in particular 
interested in high-dimensional, or nonparametric problems. It is by now well known that 
the performance of  learning or statistical methods in such settings depends crucially 
on wether or not a method succeeds in realizing the correct bias-variance trade-off, 
or, in different terminology, succeeds in balancing under- and overfitting.  
For classical, non-distributed settings we have a rather well-developed understanding of how 
methods should be tuned to achieve a proper bias-variance trade-off. 
For distributed methods however, such theory is currently not yet available.

To be able to develop relevant theory we study an idealized 
model, which is a distributed version of the canonical ``signal-in-white-noise'' model that 
serves as an important benchmark model in mathematical statistics
(see for instance \cite{tsybakov:2009, johnstone, gine:nickl:2016}). The model is on the one 
hand rich enough to be interesting, in the sense that it is really distributed in nature and the 
unknown object that needs to be learned is truly infinite-dimensional. On the other hand 
it is tractable enough to allow detailed mathematical analysis. 
In the non-distributed case the signal-in-white-noise model is well known to be 
very closely related to other nonparametric models, such as nonparametric regression and density estimation.
(This can be made very precise in the context of Le Cam's theory of limits of experiments 
(e.g.\ \cite{LeCam}, \cite{BrownLow}, \cite{Nussbaum}).) Similarly, the distributed signal-in-white-noise model that we
consider in this paper provides a unified framework to compare methods that were originally introduced 
in different settings.  We introduce the model in Section \ref{sec: model}.

It is not difficult to see that if the number of machines $m$ is relatively large 
with respect to the total sample size, or signal-to-noise-ratio $n$, then doing things 
completely naively in the distributed case leads to a sub-optimal bias-variance trade-off
(see also the simulation example in Section \ref{sec: model}). In particular, just computing the ``usual'' estimators on every local 
machine and then averaging them on the central machine typically leads to a global estimator with a 
bias that is too large. To achieve good performance, the trade-off has to be adjusted 
somehow. This can in principle be done in various ways. For instance by locally choosing 
the ``wrong'' settings for tuning parameters on purpose, or, in a Bayesian setting, by 
adjusting  the likelihood (e.g.\ raising it to some power) or  by adjusting the prior. In 
Section \ref{sec: nonadapt} we study to what degree various methods that have been proposed in the literature 
succeed in ultimately achieving the right trade-off. We will see that some are more
successful than others in this respect.

An important observation that we make is that the methods that are shown 
to work well in Section \ref{sec: nonadapt} all use information on aspects of the true signal
that are in principle unknown, such as its degree of regularity. A key question is 
whether in distributed settings it is fundamentally possible to set tuning parameters correctly in 
a purely data-driven way, without using such information. In the non-distributed setting
it is well known that such {adaptive methods} indeed exist (e.g.\ \cite{tsybakov:2009} or 
\cite{gine:nickl:2016}). In the distributed 
case that we study here however, this is much less clear. In Section \ref{sec: adap} we 
show that  using a distributed version of a standard adaptation method 
that is known to work in the 
non-distributed case, such as maximum marginal likelihood empirical  Bayes, can lead to sub-optimal 
results in the distributed setting. We will argue that this seems to be 
a fundamental issue and that we expect that correct automatic setting 
of tuning parameters in distributed methods is fundamentally more challenging 
than in the classical, non-distributed case. We believe this is an important issue  and want to highlight it as an important and interesting topic for future research.

The remainder of the paper is organized as follows. 
In the next section we introduce the distributed version of the 
signal-in-white-noise model and provide a simple simulation 
example to show that in a distributed setting, naively combining inferences 
from local machines into a global estimator may produce misleading results.
In Section \ref{sec: nonadapt} we study the performance of a number of 
 Bayesian procedures for signal reconstruction in the distributed signal-in-white-noise model
introduced in Section \ref{sec: model}. We include a number of methods
that have recently been proposed in the literature. We show that 
some succeed in obtaining the appropriate bias-variance trade-off, but others do not.
Moreover, the ones that do produce good results are all non-adaptive, in the 
sense that they use knowledge of the smoothness of the unkown signal 
to set their tuning parameters. In the final Section \ref{sec: adap} we consider 
the more realistic setting in which this smoothness is unknown. 
We study a distributed method that has been proposed for data-driven tuning of the hyperparameters
and show that there exist ``difficult signals'', which this method 
can not recover in the distributed model at an optimal rate. 
We argue that this appears to be a fundamental issue, and that 
designing procedures that automatically adapt to smoothness is fundamentally 
more challenging in the distributed framework. Mathematical proofs 
are collected in appendix Sections \ref{sec: proofs} and \ref{sec: proofs2}.

\section{Distributed signal-in-white-noise model}
\label{sec: model}

Consider the problem of estimating a signal $\theta \in \ell^2$ 
in Gaussian white noise. In the usual setting there is a single observer 
that observes every coefficient $\theta_i$ with additive Gaussian noise 
with variance $\sigma^2/n$, say. In the distributed version of the model we divide the ``precision budget'' $n$
over $m$ different observers, so that each one observes the signal in 
Gaussian noise with variance $\sigma^2m/n$, independent of the others.
In other words, observer $j$ has data $Y^j_1, Y^j_2, \ldots$ satisfying 
\begin{equation}\label{eq: local}
Y^j_i = \theta_{i} + \sqrt\frac{\sigma^2 m}{n} Z^j_i, \qquad i =1, 2, \ldots, 
\end{equation}
where the $Z^j_i$ are independent, standard Gaussian random variables.
We call the $m$ independent sub-problems in which the 
signal-to-noise ratio is $\sigma^2m/n$ the ``local'' problems.

The classical, non-distributed signal-in-white-noise model is obtained from 
the distributed model by aggregating 
all the local data. Indeed, if for $j =1, 2, \ldots$ we define $Y_i = m^{-1}\sum_{j=1}^m Y^j_i$,
then 
\begin{equation}\label{eq: agg}
Y_i = \theta_{i} + \sqrt\frac{\sigma^2}{n} \tilde Z_i, \qquad i =1, 2, \ldots,
\end{equation}
where the $\tilde Z_i = m^{-1/2} \sum_{j=1}^m Z^j_i$ are independent standard normal variables. 
This model has been studied extensively in the literature, serving as a canonical 
model for understanding the performance of high-dimensional or nonparametric statistical procedures. 
It is well known for instance that if the true signal $\theta$ belongs to an 
ellipsoid or a hyper rectangle of the form 
\[
\{\theta \in \ell^2 : \sum i^{2\beta} \theta_i^2 \le M^2\} \ \text{or} \ \{\theta\in\ell^2:\,  \sup_i( i^{1+2\beta} \theta_i^2)\leq M^2\}
\]
for some $\beta, M > 0$, then the optimal rate of convergence of estimators (relative to the $\ell^2$-norm) is of the order
$n^{-\beta/(1+2\beta)}$. Moreover, there exist so-called adaptive estimators, 
which achieve this rate without using knowledge about the parameters $\beta$ or $M$ that describe the complexity, 
or regularity of the true signal. See, for instance, \cite{tsybakov:2009} or 
\cite{gine:nickl:2016}.
Our central question is whether or not the same results can be obtained 
in the distributed setting in which each of the $m$ different observers first 
separately make  inference about the signal, and then the local  estimates 
are aggregated into one joint estimator. 

The specific examples of distributed procedures that we consider 
 in this paper are about distributed Bayesian methods. These methods have in common that  each local observer first 
 chooses a prior distribution and computes the corresponding 
local posterior distribution using the local data (or an appropriate modification). In the next step the $m$ local posteriors are somehow aggregated into a 
global posterior-type distribution, which is then used to produce an estimate 
of the signal and/or a quantification of the associated uncertainty. 
In general there is  no guarantee that this ``aggregated posterior'' resembles the 
 posterior distribution that would be obtained in the non-distributed setting, using all the data at once.
In particular, it is not clear beforehand how a distributed  Bayes method should be constructed in order
to have  good theoretical properties, like optimal convergence rates or adaptation properties. 
In this paper we investigate various  distributed  methods
that have been proposed from this point of view.

To see that interesting things can happen it is exemplifying to compare 
the results of a distributed and a non-distributed (Bayesian) analysis of simulated data. 
Concretely, we consider a  true signal $\theta$ consisting 
of the Fourier coefficients of the function shown in Figure \ref{fig: signal}.
\begin{figure}
\begin{center}
\includegraphics[scale=0.3]{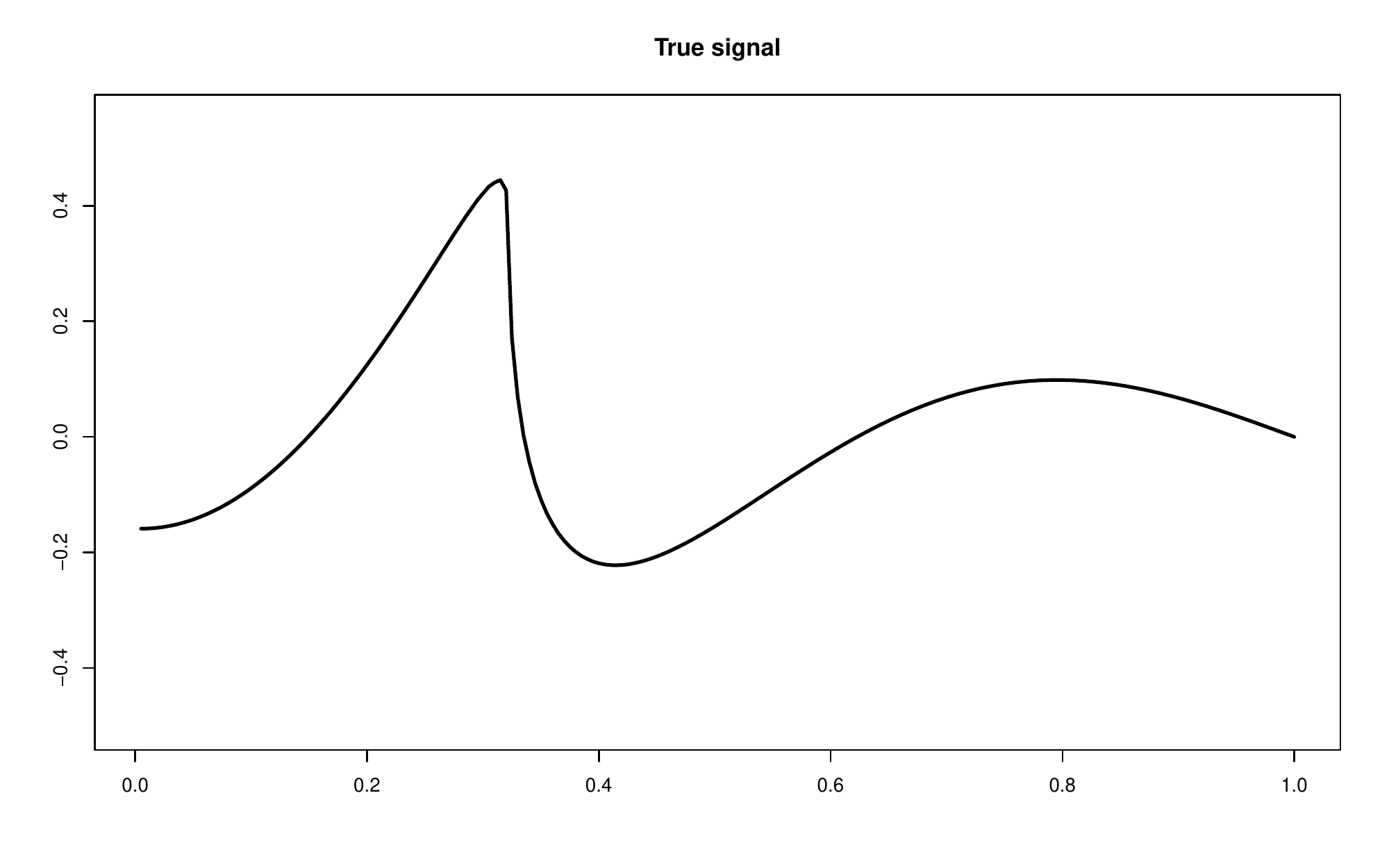}
\end{center}
\caption{True signal.}\label{fig: signal}
\end{figure}
For this signal we simulate data according to \eqref{eq: local}, with $\sigma = 1$, $m = 40$ and $n=120\times 40 = 4800$. 
Then for every local observer  a Bayesian procedure is carried out with a Gaussian prior on $\theta$, 
postulating that the coordinates $\theta_i$ are independent and $N(0, i^{-1-2\alpha})$-distributed.  
The hyperparameter $\alpha$, which describes the regularity of the prior, is determined 
using a distributed version of maximum marginal likelihood, as described in Section \ref{sec: adap}. 
This analysis leads to $m=40$ local posterior distributions. These are 
then combined to produce an overall posterior distribution for the signal. 
The precise procedure is described in Section \ref{sec: adap}. 
The resulting 
estimator for the  signal, together with pointwise $95\%$ credible intervals, is shown  in the left plot in  
Figure \ref{fig: dis}. 
\begin{figure}
\includegraphics[scale=0.3]{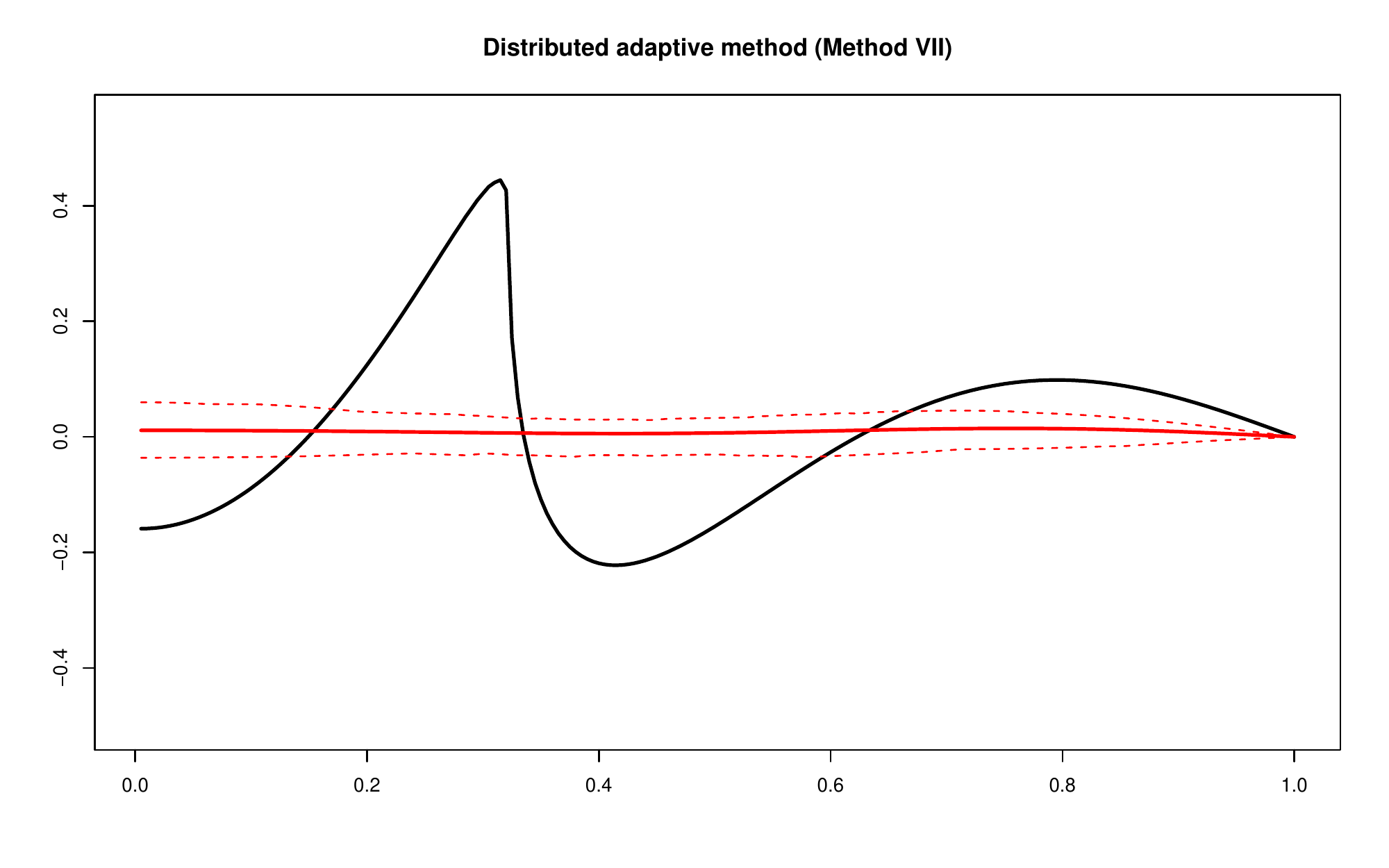}
\includegraphics[scale=0.3]{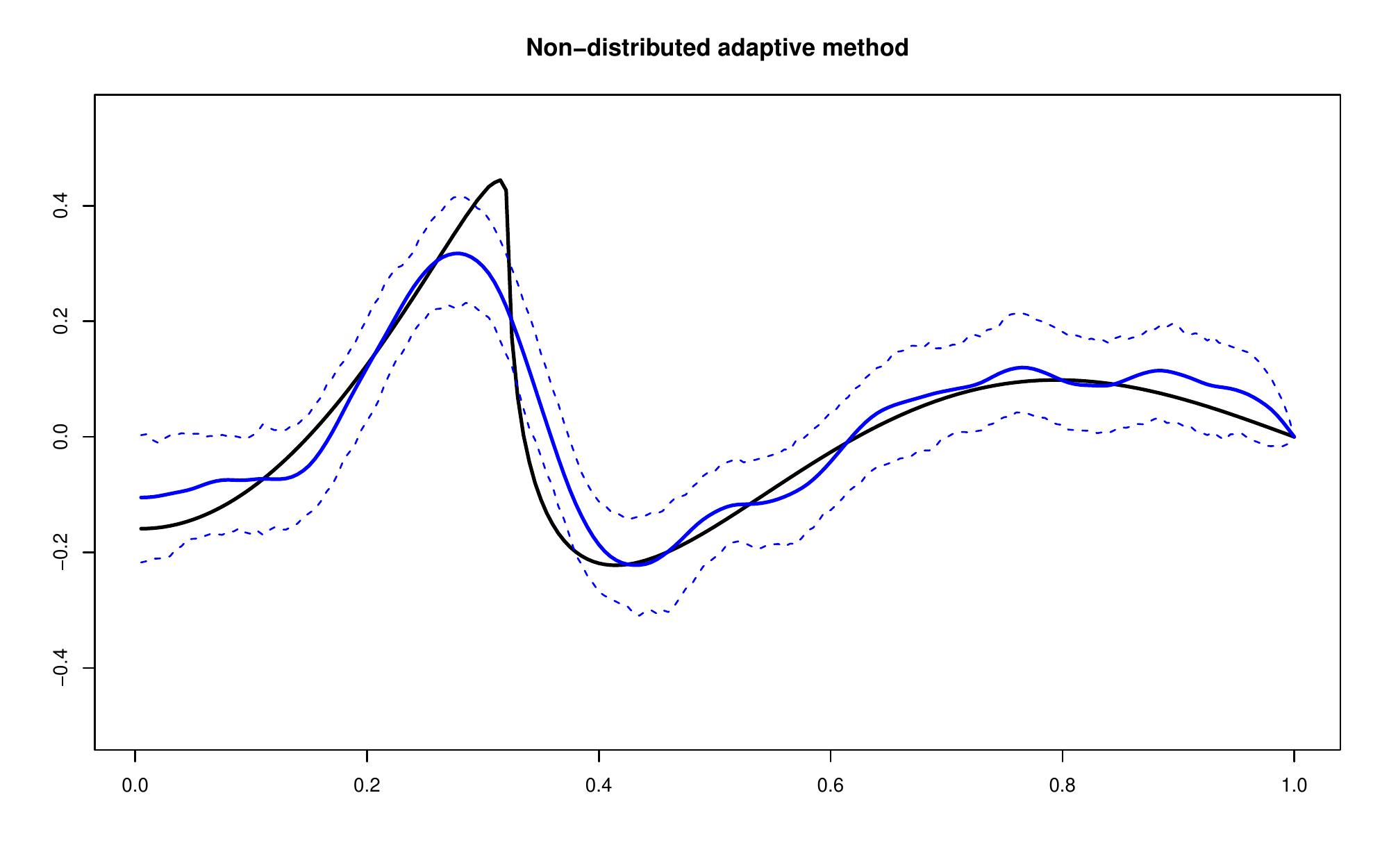}
\caption{Signal reconstruction using the distributed method (left) and the 
non-distributed method (right).}\label{fig: dis}
\end{figure}
The corresponding non-distributed result is obtained by first aggregating all local 
data as in \eqref{eq: agg} and then carrying out  the same Bayesian procedure 
 on these complete data. The resulting non-distributed reconstruction of the signal 
is shown on the right in Figure \ref{fig: dis}. 

The non-distributed version of this method was studied theoretically for instance in \cite{knapik:etal:2016}
and \cite{szabo:etal:2015}, where it was 
shown that the method is adaptive and rate-optimal. The simulation suggests however that 
an apparently reasonable  distributed analogue of the method does not necessarily inherit these 
favourable properties. The procedure seems to be underfitting and the credible intervals 
appear to be too narrow. We will argue that this is in some sense a fundamental issue and in the next sections 
we will study various proposed distributed methods to investigate to what degree
they succeed in avoiding or solving these problems.

\section{Results for non-adaptive procedures}\label{sec: nonadapt}

In this section we study the performance of a number of 
Bayesian procedures for signal reconstruction in the distributed signal-in-white-noise model
introduced in Section \ref{sec: model}. All methods involve putting a 
prior distribution  on the unknown signal $\theta \in \ell^2$ in each local problem
and then combining the resulting local posteriors into one global posterior-type
distribution. To be able to compare the various methods we consider the 
same Gaussian process (GP) prior in every case, namely the prior 
 \begin{align}
\Pi(\cdot|\a)=\bigotimes_{i=1}^{\infty} N(0, i^{-1-2\a}),\label{def: prior}
\end{align}
which postulates that the coefficients $\theta_i$ of the signal $\theta$
are independent and $N(0, i^{-1-2\a})$-distributed. The hyper parameter $\a > 0$ 
 controls the regularity of the prior. 
(Some of the methods we consider use exactly this prior, others modify it in a certain way 
with the aim of achieving better performance.) 
The global posterior-type distribution 
depends on all the data $\bY = (Y^j_i: j = 1, \ldots, m; 
i = 1,2, \ldots)$ and is denoted by $\Pi(\cdot \given  \bY)$. 
It is generally some type of average of the local posteriors, 
but its precise construction differs between proposed methods. 
We will see that this can have a significant  effect on performance.

We take an asymptotic perspective and investigate in every case the rate 
at which the global posterior contracts around the true signal as $n \to \infty$ 
relative to the $\ell^2$-norm,
which is as usual defined by $\|\theta\|^2_2 = \sum \theta_i^2$.
For a sequence of positive numbers $\eps_n \to 0$ we say that the 
global posterior contracts at the rate $\eps_n$ around the true signal $\theta_0$ if 
for all sequences $M_n \to \infty$,
\[
\EE_{\theta_0}\Pi(\theta \in \ell^2 : \|\theta -\theta_0\|_2 > M_n\eps_n\given \bY) \to 
0 
\]
as $n \to \infty$. 
This means that asymptotically, all posterior mass is 
concentrated in balls around the true signal $\theta_0$ with $\ell^2$-radius of the order $\eps_n$.

Additionally, we study how well the posterior quantifies the remaining uncertainty. 
Specifically, we  consider the coverage probabilities of credible balls around the 
global posterior mean. These credible sets are constructed by first computing 
the mean $\hat\theta$ of the global ``posterior'' $\Pi(\cdot\given \bY)$. Then for a
level $\gamma \in (0,1)$, the posterior is used to determine the radius $r_\gamma$ 
such that the ball around $\hat\theta$ with radius $r_\gamma$ receives $1-\gamma$
posterior mass, i.e.\
\[
\Pi(\theta:\, \|\theta-\hat\theta\|_2\leq r_\gamma  |\mathbf{Y})=1-\gamma.
\]
For $L > 0$, the credible set $\hat{C}(L)$ is subsequently defined by 
\begin{align}
\hat{C}(L)=\{\theta:\, \|\theta-\hat\theta_n\|_2\leq Lr_\gamma\}.
\label{def: credible}
\end{align}
(The extra constant $L$ gives some added flexibility, for $L=1$ we obtain 
an exact $1-\gamma$ credible set.)  
We are interested in the coverage probabilities $\PP_{\theta_0}(\theta_0 \in \hat C(L))$. 
If this tends to $0$ as $n \to \infty$, the credible sets are asymptotically not 
frequentist confidence sets, hence give a misleading quantification of the uncertainty. 
Ideally, the coverage probabilities stay bounded away from $0$ as $n \to \infty$.

In the non-distributed case $m=1$ it is well known that both the rate 
at which the posterior contracts around the truth and the behaviour 
of the coverage probabilities of credible sets depend crucially 
on how the hyper parameter $\a$ is tuned. 
The correct bias-variance trade-off is achieved if $\alpha$ 
is in accordance with the regularity of the unknown signal. 
To make this precise, we will consider signals belonging to hyper rectangles
of the form
\begin{equation}\label{eq: hb}
H^\beta(M) = \Big\{\theta \in \ell^2: \sup_i (i^{1+2\beta}\theta^2_i) \le M^2\Big\}
\end{equation}
for some $\beta, M > 0$.  
It is shown for instance in \cite{knapik:etal:2011} for the non-distributed case
that if $\theta_0 \in H^\beta(M)$ 
and  we set $\alpha = \beta$, then the posterior contracts around $\theta_0$ 
at the optimal rate $n^{-\beta/(1+2\beta)}$. Moreover, 
for $L$ large enough  
it then holds that $\PP_{\theta_0}(\theta_0 \not\in \hat{C}(L)) \le \gamma$. 
Hence, in the non-distributed case it is optimal to choose the hyper parameter 
$\a$ in such a way that the regularity $\alpha$ of the prior matches the regularity 
$\beta$ of the true signal. Moreover, this choice leads to a  contraction rate that is optimal 
in a minimax sense.

In the remainder of this section we investigate distributed methods from  
this point of view. We will see that the different proposed methods lead 
to different behaviours in terms of contraction rates and coverage. We 
stress that the results in this section are non-adaptive, in the sense
that we allow the tuning parameter $\alpha$ and other aspects of the constructions
to use knowledge of the regularity $\beta$ of the true signal. 
This is of course not realistic. It is however important to first 
understand for every method whether ideally, if the value of $\beta$ is given to us by an oracle, it 
is possible to tune the method optimally. Whether this is also possible
adaptively, without knowing $\beta$, is then the next natural question, 
which we address in Section \ref{sec: adap}.

\subsection{Naive  averaging of local posterior draws}\label{sec: naive_approaches}

Recall that we have $m$ local observers that each have a dataset $\mathbf Y^j = (Y^j_1, Y^j_2, \ldots)$ of noisy
 coefficients satisfying \eqref{eq: local}. The aim is to recover the true sequence
of coefficients $\theta$. 

As a starting point, and to have a baseline case to compare the other methods to, 
we analyse the  naive distributed approach in which in every local 
problem we simply use the prior $\Pi(\cdot \given \alpha)$ defined by \eqref{def: prior}, 
with $\alpha = \beta$ equal to the regularity of the true sequence $\theta$, in the sense of \eqref{eq: hb}.  
Every local observer then computes its corresponding local posterior, $\Pi^j(\cdot \given \mathbf Y^j)$. 
By Bayes' formula this is given by 
\[
d\Pi^j(\theta \given \mathbf Y^j) \propto p(\mathbf Y^j \given \theta)\,d\Pi(\theta \given \beta), 
\]
where the  likelihood for the $j$th local problem is given by 
\begin{equation}\label{eq: lik}
p(\mathbf Y^j \given \theta)  \propto \prod e^{-\frac12 \frac{n(Y^j_i-\theta_i)^2}{\sigma^2 m}}.
\end{equation}
Finally these local posteriors are combined into a  global, ``average posterior'' ${\mathbf\Pi_{I}(\cdot \given \mathbf Y)}$ by postulating that a draw
from this global posterior is generated by first drawing once from each local posterior
and then averaging these $m$ independent draws. (Formally, this means that the global
``posterior''
${\mathbf\Pi_{I}(\cdot \given \mathbf Y)}$ is the convolution of the rescaled 
local posteriors 
$\Pi^1(m \times \cdot\given \mathbf Y^1)$, \ldots, $\Pi^m(m \times \cdot\given \mathbf Y^m)$).  

This distributed  method is conceptually very simple, but it turns out that neither from the point of view of contraction rates, 
nor from the point of view of uncertainty quantification it performs very well. The reason 
is basically that although the choice $\alpha = \beta$ of the tuning parameter of the prior correctly matches 
squared bias, variance and posterior spread in the local problems, the averaging procedure results in 
a global ``posterior'' for which the spread and the variance of the mean are too small relative to the squared bias. 
The following theorem asserts that for every smoothness level $\beta > 0$ 
there exist $\beta$-regular truths for which the 
  contraction rate of the posterior deteriorates substantially and for which the uncertainty quantification 
by the credible sets \eqref{def: credible} constructed from the global posterior 
${\mathbf\Pi_{I}(\cdot \given \mathbf Y)}$ is useless, 
no matter how far they are blown up by a constant $L > 0$. 

\bigskip

\begin{theorem}[naive averaging]\label{thm: naive_avg}
For every $\beta, M>0$ there exists a $\theta_0\in H^{\beta}(M)$ such that for  small enough  $c>0$,
\begin{align*}
E_{\theta_0}\mathbf\Pi_{I}(\theta:\, \|\theta-\theta_0\|_2\leq cm^{\frac{\beta}{1+2\beta}}n^{-\frac{\beta}{1+2\beta}}|\beta,\mathbf{Y} )\rightarrow 0
\end{align*}
as $m \to \infty$ and $n/m \to \infty$. 
Furthermore, for all $L>0$ it holds that 
\[
P_{\theta_0}\big(\theta_0\in \hat{C}(L) \big)\rightarrow 0.
\]
\end{theorem}

\begin{proof}
The proof of the theorem is given in  Section \ref{sec: naive_avg}.
\end{proof}

\bigskip

In the literature several less naive distributed strategies have been proposed. These 
methods either change the local likelihoods in a certain way, and/or the priors 
that are locally used, and/or the way that the local posteriors are aggregated. 
In the next few sections we investigate whether such strategies can improve
the bad asymptotic performance of the naive averaging method.

\subsection{Adjusted local  likelihoods and averaging}\label{sec: likelihood_approaches}

One perspective on the bad performance of the naive method is to say that since 
the ``sample size'' $n/m$ in the local problems is too small, the influence of the data on the 
local posterior is too small, resulting in a variance (and spread) that is too small 
relative to the squared bias. A possible way to remedy this that has been proposed in several papers is to 
 raise the local likelihoods to the power $m$, in order to mimic the situation that we 
have sample size $n$ in the local problems. This generalized Bayesian approach for the local 
problems has for instance been considered in the distributed context by  \cite{srivastava:2015a}. 
They combine it with a different aggregation method however, which we consider in Section 
\ref{sec: was}.
In this section we still consider the simple averaging scheme, in order
to isolate the effect of adjusting the local likelihoods. 

So in method II all local observers use the prior $\Pi(\cdot \given \alpha)$ again, 
with $\alpha = \beta$ equal to the regularity of the truth. They now each compute
a generalized local posterior $\tilde \Pi^j(\cdot \given \mathbf Y^j)$, defined by  
\[
d\tilde \Pi^j(\theta \given \mathbf Y^j) \propto \Big(p(\mathbf Y^j \given \theta)\Big)^m\,d\Pi(\theta \given \beta). 
\]
As before the  global ``posterior'' ${\mathbf\Pi_{II}(\cdot \given \mathbf Y)}$ is defined by postulating that a draw
from this global posterior is generated by first drawing once from each local generalized posterior
and then averaging these $m$ independent draws.

The following theorem states that this method indeed improves the naive approach of Section 
\ref{sec: naive_approaches}.
The  global posterior now contracts at the optimal rate for every $\beta$-regular truth. 
Unfortunately, the bad behaviour of the credible sets has not been remedied. For this 
approach the uncertainty quantification is in fact misleading 
for all $\beta$-regular truths.

\bigskip

\begin{theorem}[adjusted likelihoods + averaging]\label{thm: likelihood_average}
For all $\beta, M>0$ and  all sequences $M_n \to \infty$, 
\begin{align*}
\sup_{\theta_0 \in H^\beta(M)}\EE_{\theta_0}\mathbf\Pi_{{II}}(\theta:\, \|\theta-\theta_0\|_2\ge M_n n^{-\frac{\beta}{1+2\beta}}|\mathbf{Y} )\rightarrow 0 
\end{align*}
as $n, m \to \infty$. 
However, for all $\theta_0 \in H^\beta(M)$ and all $L>0$ it holds that  
\[
\PP_{\theta_0}\big(\theta_0\in \hat{C}(L) \big)\rightarrow 0.
\]
\end{theorem} 

\begin{proof}
The proof is given in Section \ref{sec: likelihood_average}.
\end{proof}

\bigskip

\subsection{Adjusted priors and averaging}\label{sec: prior_approaches}

Adjusting the likelihood as in the preceding section resulted in 
a correct trade-off between the bias and the variance of the global posterior mean, 
yielding an optimal posterior contraction rate. The spread of the posterior 
remained  too small in comparison however, resulting in credible sets  
with zero asymptotic coverage.
Instead of raising the local posteriors to the power $m$, as considered in the 
preceding section, we could alternatively raise the prior density to the 
power $1/m$. This has for instance been proposed in the context of the 
``Consensus Monte Carlo''  approach by \cite{scott:etal:2013}, in combination with 
simple averaging of the local posteriors. In this section we investigate the performance of this 
method in terms of posterior contraction and uncertainty quantification in our distributed 
signal-in-white-noise model.

The prior $\Pi(\cdot \given \alpha)$  that we use in the local problems is again a
product of centered Gaussians
with variance $i^{-1-2\alpha}$. Raising the corresponding densities to the power $1/m$ 
has the effect of multiplying  the $i$th prior variance by $m$. Hence, 
in our case raising the prior density to the power $1/m$ is the same as multiplicative 
rescaling, postulating that  $\theta$ is a-priori distributed according to $\Pi(\cdot \given \alpha, m)$, 
where 
\begin{align}
\label{eq: pi2} 
\Pi(\cdot |\a,\t)=\bigotimes_{i=1}^{\infty} N(0, \t i^{-1-2\a})
\end{align}
for $\a, \tau > 0$. Rescaled GPs have also been considered by \cite{shang:cheng:2015}, who have used them in the 
distributed setting to construct global credible sets from local ones.

Using rescaling we can actually obtain good results if the  prior regularity $\alpha$ is not 
exactly equal to the true regularity $\beta$. By using a scaling different from $\tau = m$ 
we can somehow compensate for the mismatch between $\alpha$ and $\beta$, at least in the range
$\beta \le 1+2\alpha$. In the non-distributed setting this is a well-known phenomenon, 
see for instance \cite{vaart:zanten:2007, knapik:etal:2011, szabo:etal:2013}.

The distributed procedure that we consider in this section then takes the following form. 
Every local observer uses the rescaled prior $\Pi(\cdot |\a,\t)$ defined by \eqref{eq: pi2}, with 
$\alpha > 0$ and 
\[
\tau = m n^{\frac{2(\alpha - \beta)}{1+2\beta}},
\]
where $\beta$ is the regularity of the truth. Next the (normal, unadjusted) corresponding posteriors are 
computed and they are averaged into a global ``posterior'' $\mathbf\Pi_{III}(\cdot\given \mathbf Y)$
as in the preceding sections. (Note that if in the local problems the prior regularity $\alpha = \beta$
is used, then $\tau = m$, so the method corresponds to raising the prior density to the power $1/m$.)

The following theorem gives the posterior contraction and coverage results 
for this method. 

\bigskip

\begin{theorem}[adjusted priors + averaging]\label{thm: prior_average}
Suppose  $\beta, M>0$ and  $\beta \le 1 + 2\alpha$. Then for all sequences $M_n \to \infty$, 
\begin{align*}
\sup_{\theta_0 \in H^\beta(M)} \EE_{\theta_0}\mathbf\Pi_{III}(\theta:\, \|\theta-\theta_0\|_2 > 
M_n n^{-\frac{\beta}{1+2\beta}}|\mathbf{Y} )\rightarrow 0
\end{align*}
as $n \to \infty$. Moreover, for all $\gamma \in (0,1)$   it holds that
\begin{align*}
\sup_{\theta_0 \in H^\beta(M) }\PP_{\theta_0}\Big(\theta_0 \not\in \hat{C}(L) \Big)\le \gamma
\end{align*}
for large enough $L>0$.
\end{theorem}

\begin{proof}
See  Section \ref{sec: prior_average}.
\end{proof}
 
\bigskip 
 
So adjusting the prior in this way actually works better 
than adjusting the likelihood. Not only do we get optimal contraction rates, 
but the  credible sets that this method produces have asymptotic coverage
too. The proof shows that the credible sets have optimal radius of the order $n^{-\beta/(1+2\beta)}$
as well.

\subsection{Adjusted local likelihoods and Wasserstein barycenters}
\label{sec: was}

In Section \ref{sec: likelihood_approaches} we saw that raising the local likelihoods to the power $m$ and then 
averaging the corresponding generalized posteriors yields optimal contraction 
rates, but can produce badly performing credible sets. In this section we study
 the approach considered by \cite{minsker:2014,srivastava:2015a}  in the context
 of their ``WASP'' method, which consists in aggregating the local posteriors not by 
 simple averaging, but by  computing their Wasserstein barycenter.  

The generalized local posteriors $\tilde \Pi^j(\cdot \given \mathbf Y^j)$, as defined in Section 
\ref{sec: likelihood_approaches}, are (Gaussian) measures on $\ell^2$. The $2$-Wasserstein distance $W_2(\mu, \nu)$
between two probability measures $\mu$ and $\nu$ on $\ell^2$ is defined by 
\[
W^2_2(\mu, \nu) = \inf_\gamma \int\!\!\int \|x-y\|_2^2\,\gamma(dx, dy),
\]
where the infimum is over all measures $\gamma$ on $\ell^2 \times \ell^2$ with marginals $\mu$ and $\nu$. 
The corresponding 2-Wasserstein barycenter of $m$ probability measures $\mu_1, \ldots, \mu_m$ on $\ell^2$ is then 
defined by 
\[
\bar \mu = \argmin_\mu \frac1m\sum_{j=1}^m W^2_2(\mu, \mu_j), 
\]
where the minimum is over all probability measures on $\ell^2$ with finite second moments. There exist 
effective algorithms to compute Wasserstein barycenters in many cases, 
see for instance \cite{cuturi14} and the references therein.

Having this notion at our disposal the distributed method we consider in this section 
proceeds as follows. In every local problem the prior $\Pi(\cdot\given \alpha)$ 
is used, with $\alpha=\beta$ equal to the regularity of the truth. 
Next, the corresponding generalized posteriors 
$\tilde \Pi^j(\cdot \given \mathbf Y^j)$ are computed locally, which involves 
raising the likelihood to the power $m$
as described in Section \ref{sec: likelihood_approaches}.
Finally, the global ``posterior'' $\mathbf \Pi_{IV}(\cdot\given \mathbf Y)$
is constructed as the 2-Wasserstein barycenter
of the local measures $\tilde \Pi^1(\cdot \given \mathbf Y^1), 
\ldots, \tilde \Pi^m(\cdot \given \mathbf Y^m)$.

The following theorem asserts that this method results in optimal 
posterior contraction rates and correct quantification of uncertainty.

\bigskip

\begin{theorem}[adjusted likelihoods + barycenters]\label{thm: likelihood_Wasserstein}
For all  $\beta, M>0$ and   all sequences $M_n \to \infty$, 
\begin{align*}
\sup_{\theta_0 \in H^\beta(M)} \EE_{\theta_0}\mathbf\Pi_{IV}(\theta:\, \|\theta-\theta_0\|_2 > 
M_n n^{-\frac{\beta}{1+2\beta}}|\mathbf{Y} )\rightarrow 0
\end{align*}
as $n \to \infty$. Moreover, for all $\gamma \in (0,1)$   it holds that
\begin{align*}
\sup_{\theta_0 \in H^\beta(M) }\PP_{\theta_0}\Big(\theta_0 \not\in \hat{C}(L) \Big)\le \gamma
\end{align*}
for large enough $L>0$.
\end{theorem}

\begin{proof}
See Section \ref{sec: likelihood_Wasserstein}.
\end{proof}

\bigskip

\subsection{Product of Gaussian process experts}

The proofs of the theorems presented so far show that since in our 
context the global ``posterior'' is always a Gaussian measure, 
the behaviour of the procedure can be understood by analyzing  
three central quantities: the bias of the posterior mean, the variance 
of the posterior mean, and the spread of the posterior. Depending
on how these quantities are related we have found different behaviours:
sub-optimal posterior contraction and bad coverage of credible sets 
(Section \ref{sec: naive_approaches}), optimal posterior contraction but 
bad coverage of credible sets (Section \ref{sec: likelihood_approaches}), 
and optimal posterior contraction 
and also good coverage of credible sets (Sections \ref{sec: prior_approaches} 
and \ref{sec: was}).

In principle it is now straightforward to analyze different methods as well, 
provided the three central quantities can be controlled. As an illustration 
we consider in this section the single-layer version of the 
product-of-Gaussian-process-expert (PoE) model, introduced in \cite{ng:2014} and a generalization proposed  in \cite{cao:2014}. 
An interesting fact is that we will encounter
a combination of behaviours that we have not seen yet: sub-optimal 
contraction rates, but good coverage of credible sets. 
 These methods were introduced to deal with the distributed non-parametric regression model, but 
for the sake of comparison we analyze them in the context of our distributed signal-in-white-noise model, which can be thought of as an idealized version of the regression model.

The idea of the basic version of the Gaussian PoE model is to employ a Gaussian prior in 
every local machine, compute the corresponding posterior densities
and approximate the global posterior density by multiplying  and normalizing
these. In our infinite-dimensional setting this does not make sense strictly 
speaking, since we can not express priors and posteriors on $\ell^2$
in terms of densities with respect to some generic dominating measure. 
We could remedy this by considering a truncated version of our distributed 
model, where we assume we only observe the first $n$ noisy coefficients $Y^j_i$ in 
every machine, say, and focus on making inference about the first $n$ true 
coefficients $\theta_i$. This would make the setting finite-dimensional, allowing 
us to write prior and posterior densities with respect to Lebesgue measure. 
Alternatively, we can stay in the infinite-dimensional setting of the paper 
and just reason formally and still arrive at a well-defined 
global PoE ``posterior''. This is the approach we follow here.

Indeed, say that as before we use the prior $\Pi(\cdot \given \alpha)$ given by \eqref{def: prior}
in every  local machine, with $\alpha=\beta$ equal
to the  regularity of the true signal. This prior has formal ``density'' proportional to
\[
\theta \mapsto \prod  e^{-\tfrac12 \frac{\theta_i^2}{i^{-1-2\beta}}}.
\]
By completing the square we see that the product of this expression with the local likelihood given by \eqref{eq: lik}  is,
still formally, proportional to 
\[
\theta \mapsto \prod  e^{-\tfrac12 \frac{\theta_i^2}{i^{-1-2\beta}}} e^{-\frac12 \frac{n(Y^j_i-\theta_i)^2}{\sigma^2 m}} \propto \prod 
e^{-\frac12 \theta_i^2(i^{1+2\beta} + \frac{n}{m\sigma^2}) 
+\theta_i \frac{nY^j_i}{m\sigma^2}}.
\]
Taking the product over $j$ we then obtain the formal density of the 
PoE posterior, which is proportional to 
\[
\theta \mapsto \prod 
e^{-\frac12 \theta_i^2(mi^{1+2\beta} + \frac{n}{\sigma^2}) 
+\theta_i \frac{n\sum_{j =1}^mY^j_i}{m\sigma^2}}.
\]
Now this last expression is, up to a constant, the density of 
a product of Gaussians with means $\hat\theta_i$ and variances $t^2_i$ 
given by 
\[
\hat\theta_i = \frac{nm^{-1}\sum Y^j_i}{n + \sigma^2mi^{1+2\beta}}, 
\qquad t^2_i = \frac{\sigma^2}{n + \sigma^2mi^{1+2\beta}} .
\]
The latter is in fact a well-defined Gaussian measure on $\ell^2$, 
so we can now simply define the global PoE ``posterior'' 
$\mathbf \Pi_{V}(\cdot\given \mathbf Y)$ as the latter measure. 

We see that  the expressions 
for the global mean and spread are in fact the same as what we found 
 in Section \ref{sec: naive_avg} for the naive averaging method. 
 As a consequence, the negative result of Theorem \ref{thm: naive_avg}
 holds for the basic version of the Gaussian PoE model as well.
 
 \bigskip
 
\begin{theorem}[product of Gaussian experts]\label{thm: poe}
For every $\beta, M>0$ there exists a $\theta_0\in H^{\beta}(M)$ such that for  small enough  $c>0$,
\begin{align*}
\EE_{\theta_0}\mathbf\Pi_{V}(\theta:\, \|\theta-\theta_0\|_2\leq cm^{\frac{\beta}{1+2\beta}}n^{-\frac{\beta}{1+2\beta}}| \mathbf{Y} )\rightarrow 0
\end{align*}
as $m \to \infty$ and $n/m \to \infty$. 
Furthermore, for all $L>0$ it holds that 
\[
\PP_{\theta_0}\big(\theta_0\in \hat{C}(L) \big)\rightarrow 0.
\]
\end{theorem}

\bigskip

 One can generalize the PoE model by raising the local posterior 
 densities to some power before multiplying and normalizing them,
 as proposed in \cite{cao:2014}. 
 In the subsequent  analysis we consider the choice 
 to $1/m$ for the power, as suggested in \cite{deisenroth:2015}.
 Adapting the preceding analysis for the ordinary PoE model we see that 
for this generalized  PoE model the global ``posterior'' 
$\mathbf \Pi_{VI}(\cdot\given \mathbf Y)$ is in our setting 
again a product of Gaussian, but now with means and variances given by 
\[
\hat\theta_i = \frac{nm^{-1}\sum Y^j_i}{n + \sigma^2mi^{1+2\beta}}, 
\qquad t^2_i = \frac{\sigma^2m}{n + \sigma^2mi^{1+2\beta}} .
\]
So the global posterior mean is unaltered compared to the basic PoE model, 
but the global posterior spread 
has been blown up by a factor $m$. As a result, there still 
exists the same class of truths as in Section 
\ref{sec: naive_avg}  for which the squared bias and the variance 
of the posterior mean will be incorrectly balanced, resulting 
in a sub-optimal rate of posterior contraction. However, 
the larger posterior spread ensures that we do have asymptotic 
coverage of credible sets. It should be noted however that 
 these sets have a
diameter that is sub-optimal, i.e.\ they are too conservative. 

\bigskip

\begin{theorem}[generalized product of Gaussian experts]\label{thm: gpoe}
For every $\beta, M>0$ there exists a $\theta_0\in H^{\beta}(M)$ such that for  small enough  $c>0$,
\begin{align*}
\EE_{\theta_0}\mathbf\Pi_{VI}(\theta:\, \|\theta-\theta_0\|_2\leq cm^{\frac{\beta}{1+2\beta}}n^{-\frac{\beta}{1+2\beta}}|\mathbf{Y} )\rightarrow 0
\end{align*}
as $m \to \infty$ and $n/m \to \infty$. 
However, for all $\gamma \in (0,1)$   it holds that
\begin{align*}
\sup_{\theta_0 \in H^\beta(M) }\PP_{\theta_0}\Big(\theta_0 \not\in \hat{C}(L) \Big)\le \gamma
\end{align*}
for large enough $L>0$.
\end{theorem}

\begin{proof}
The proof of the theorem can be found in Section \ref{sec: gpoe}.
\end{proof}

\bigskip

\subsection{Summary of results for non-adaptive methods}
\label{sec: summary}

We have seen that the various methods for aggregation
of the local posteriors can give quite different results. 
The methods we considered produce  different global 
``posterior'' measures. 
Depending on the relation between the bias and variance 
of the global posterior mean and the spread of this 
global posterior, the posterior contraction rate 
and coverage probabilities of credible sets can have different behaviours. 
We summarize our findings in Table \ref{tab: sum}.
This is certainly not meant to be an exhaustive list of methods, 
but rather  an illustration of how the design
of distributed procedures can affect their fundamental performance.

\begin{table}
\begin{center}
\begin{tabular}{|llcc|}
\hline
Method & Description & Optimal rate & Coverage\\
\hline
\hline
I &  naive averaging & no & no\\
II & adjusted likelihoods, averaging & yes & no\\
III & adjusted priors, averaging & yes & yes\\
IV & adjusted likelihoods, barycenter & yes & yes\\
V & product of experts & no & no\\
VI & generalized product of experts & no & yes\\
\hline
\end{tabular}
\caption{Performance of the various non-adaptive methods.}
\label{tab: sum}
\end{center}
\end{table}

Simulations further illustrate the theoretical results. 
We have considered a  true signal $\theta$ consisting 
of the Fourier coefficients of the function shown in the left panel of Figure \ref{fig: signal2}.
This is a signal which has regularity $\beta = 1$ in the sense of \eqref{eq: hb}. 
For this signal we simulated data according to \eqref{eq: local}, with $\sigma = 1$, $n=4800$ and $m = 40$, 
i.e.\ we considered a distributed setting with $m =40$ machines. 
For the sake of comparison, the right panel of Figure \ref{fig: signal2} shows the signal reconstruction
and uncertainty quantification for the non-distributed method which first aggregates all 
data in a single machine and then computes the posterior corresponding to the prior 
$\Pi(\cdot\given \alpha)$ defined by \eqref{def: prior}, with $\alpha = \beta$.
This is a method which is known to have an optimal convergence rate and correct quantification of uncertainty.
This classical, non-distributed result should be compared to Figure \ref{fig: dis2}, which  visualizes the ``posteriors'' generated by each of the 
distributed methods I--VI. 

In accordance with our theoretical results, we see that 
the results of methods III and IV are comparable with the non-distributed method. 
Methods I, V and VI have worse signal reconstruction. The posterior mean of Method II is comparable to that of the optimal methods, 
but the uncertainty is underestimated.

\begin{figure}
\includegraphics[scale=0.3]{true.pdf}
\includegraphics[scale=0.3]{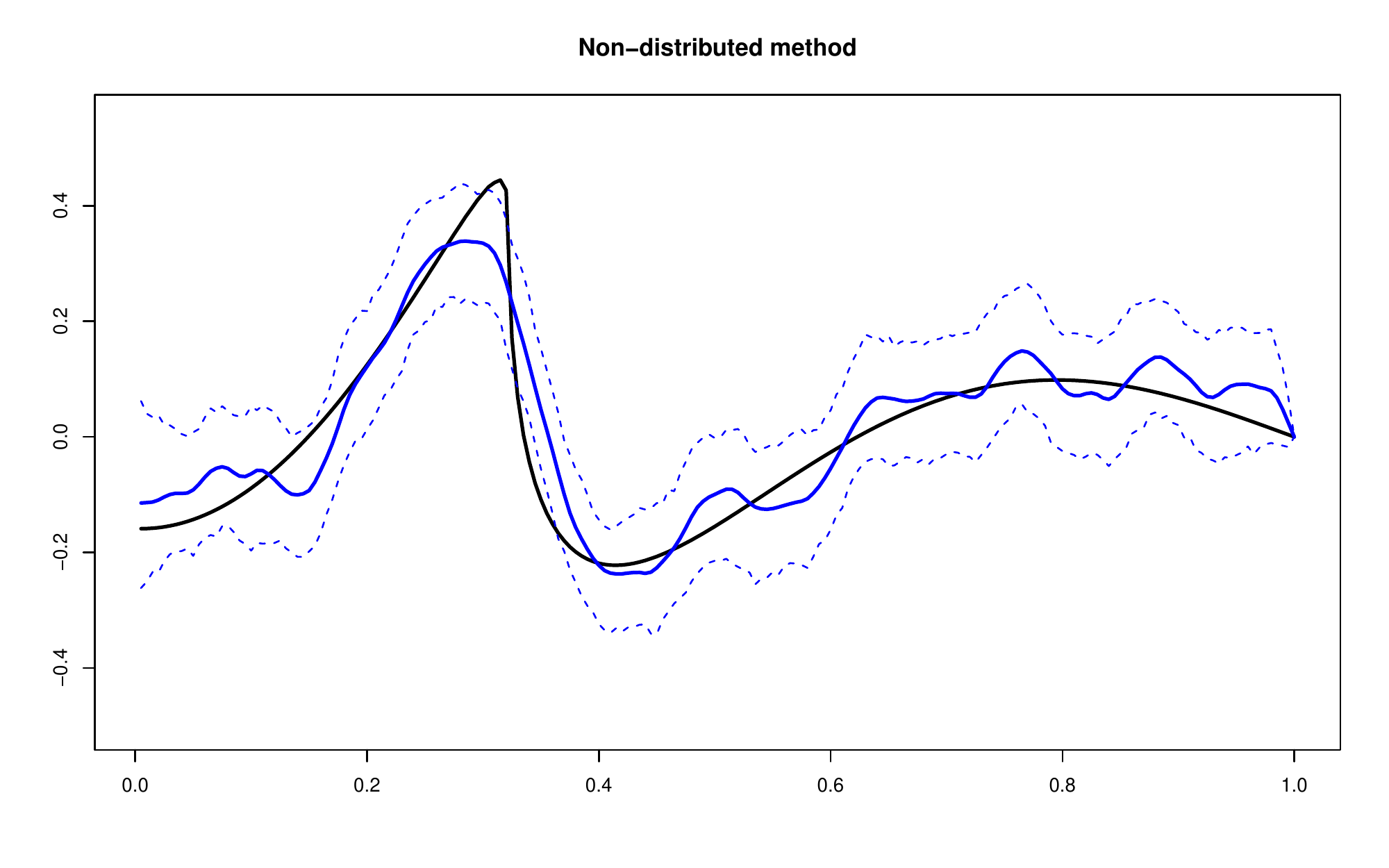}
\caption{Left: true signal. Right: posterior mean (blue solid curve)  and $95\%$ pointwise credible bands (dashed blue curves) for
the non-distributed method.}\label{fig: signal2}
\end{figure}

\begin{figure}
\includegraphics[scale=0.3]{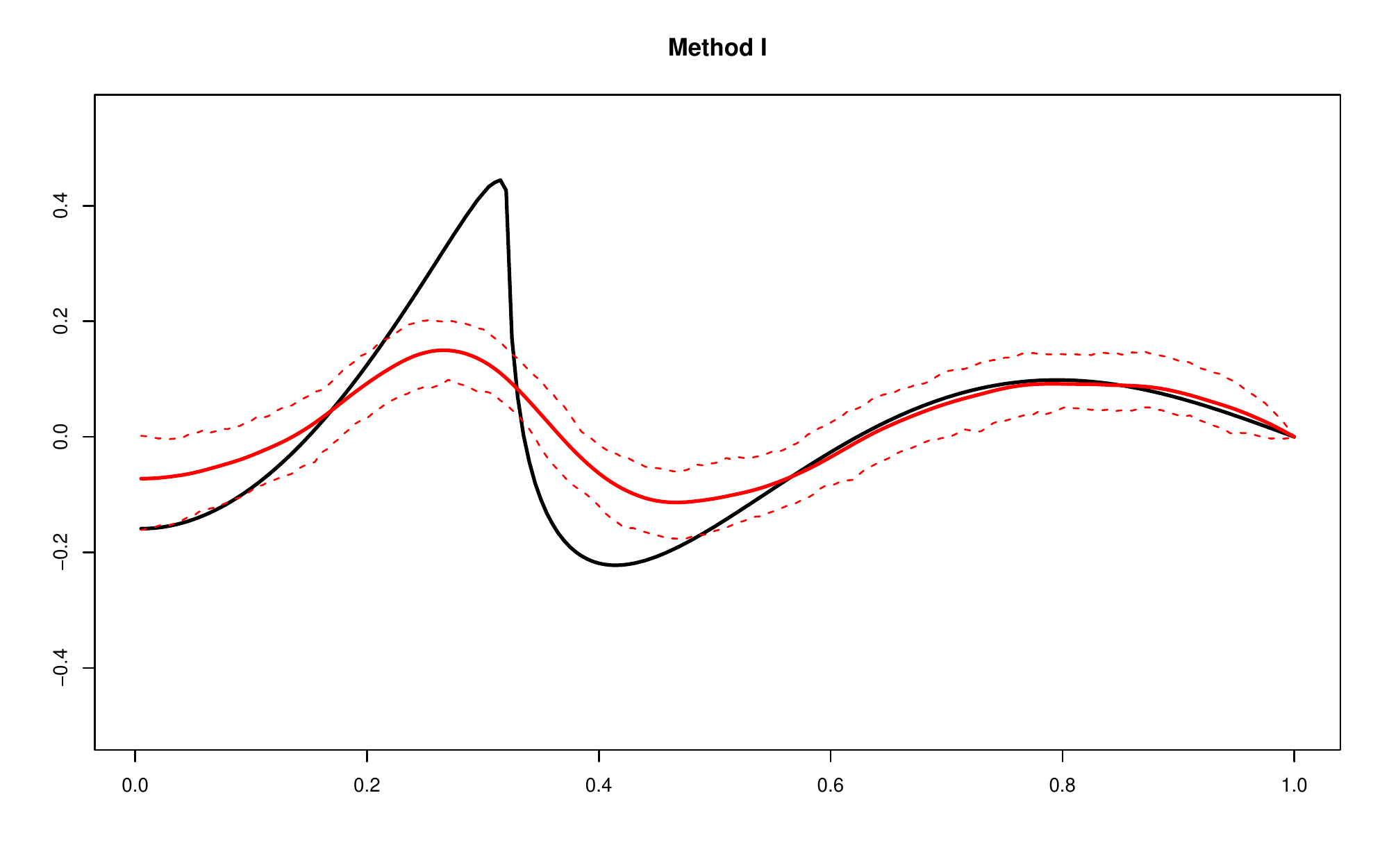}
\includegraphics[scale=0.3]{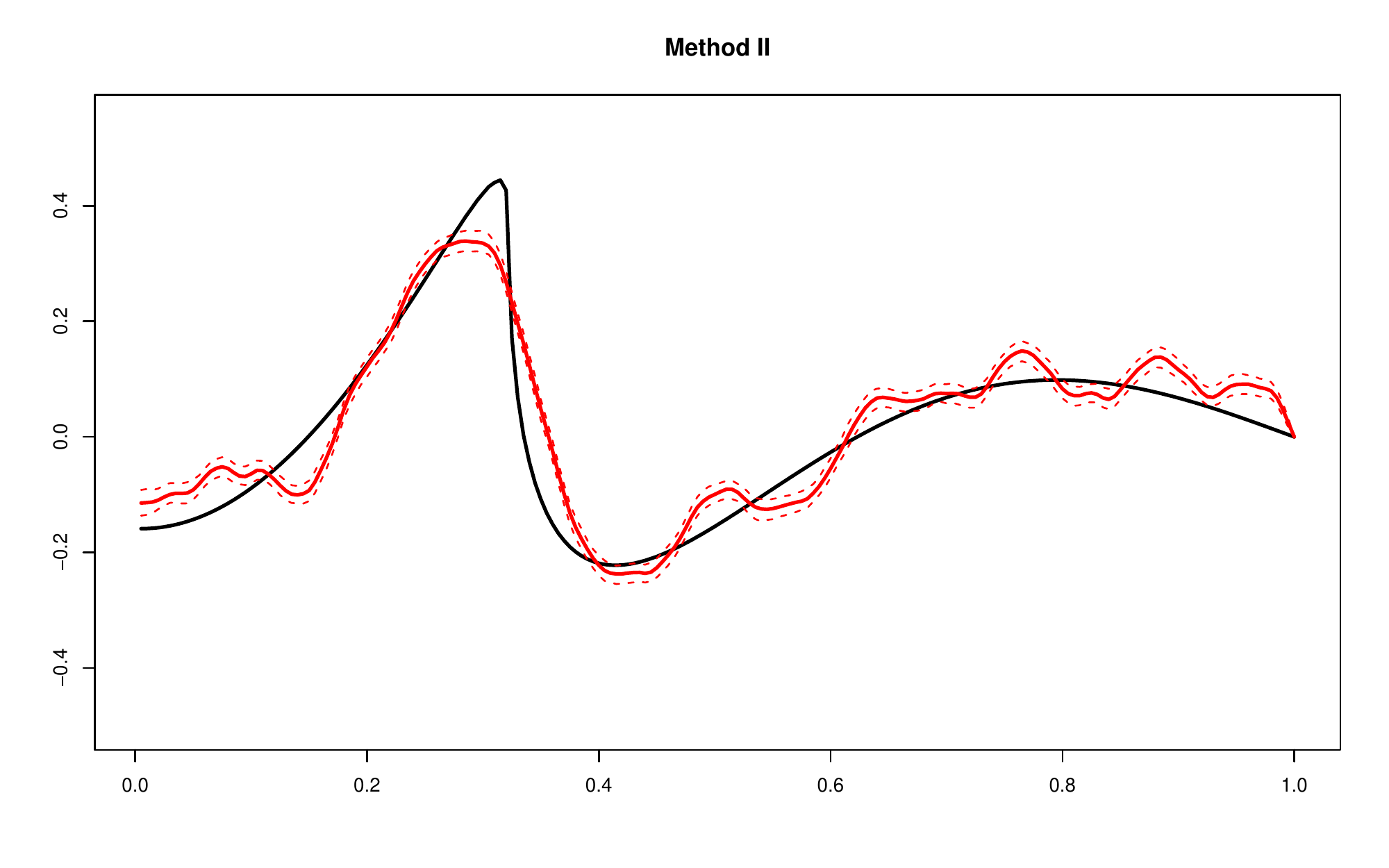}
\includegraphics[scale=0.3]{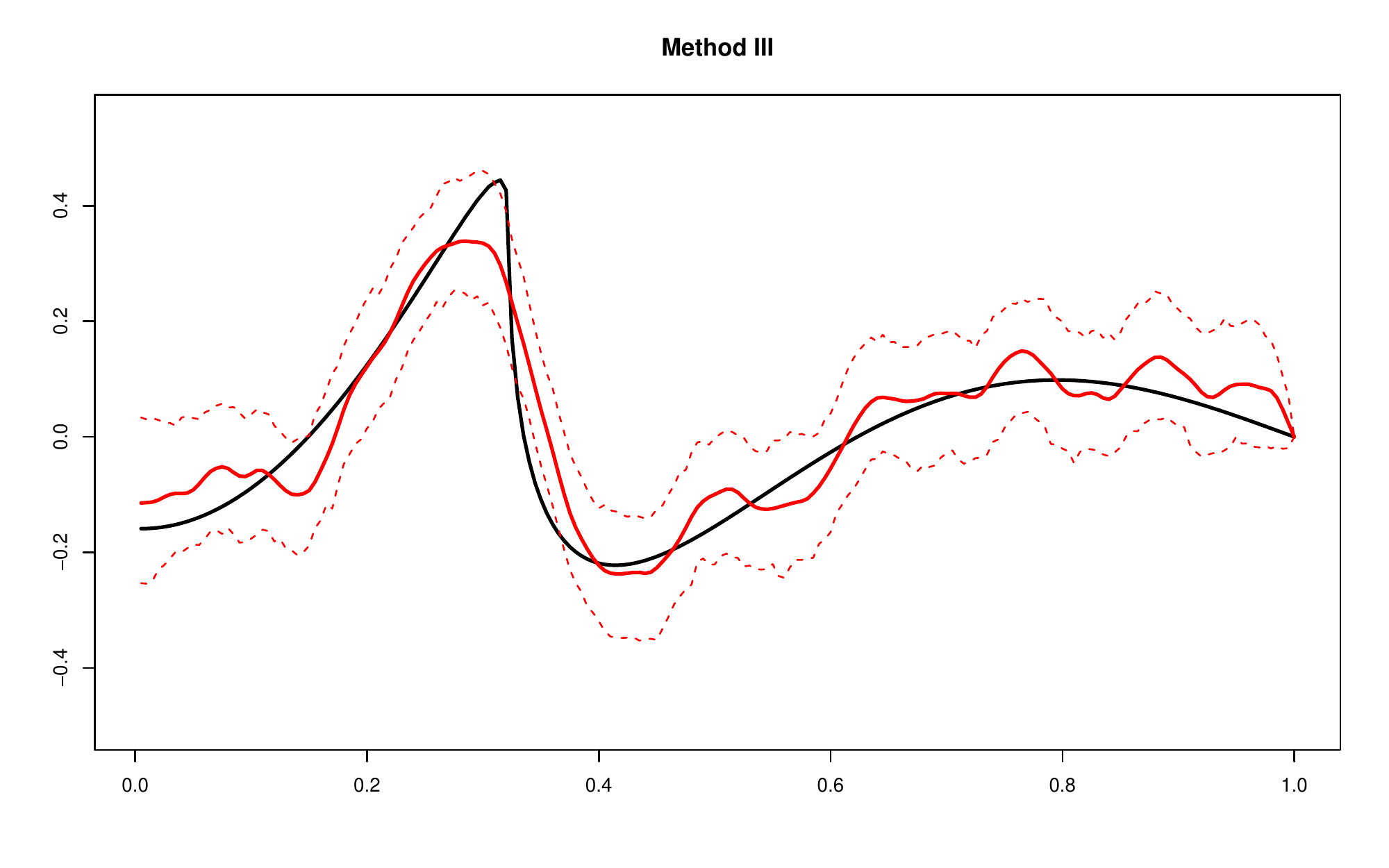}
\includegraphics[scale=0.3]{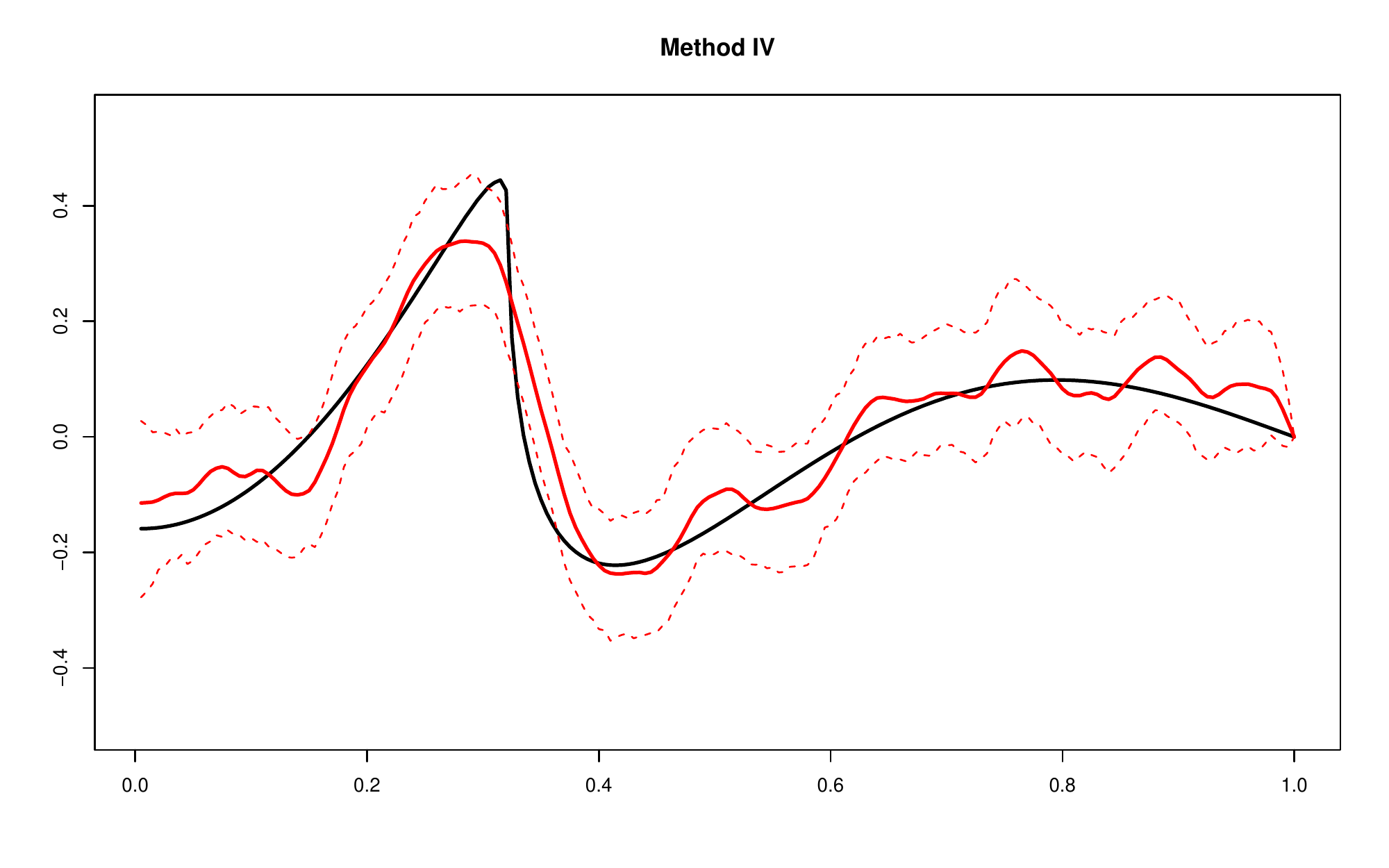}
\includegraphics[scale=0.3]{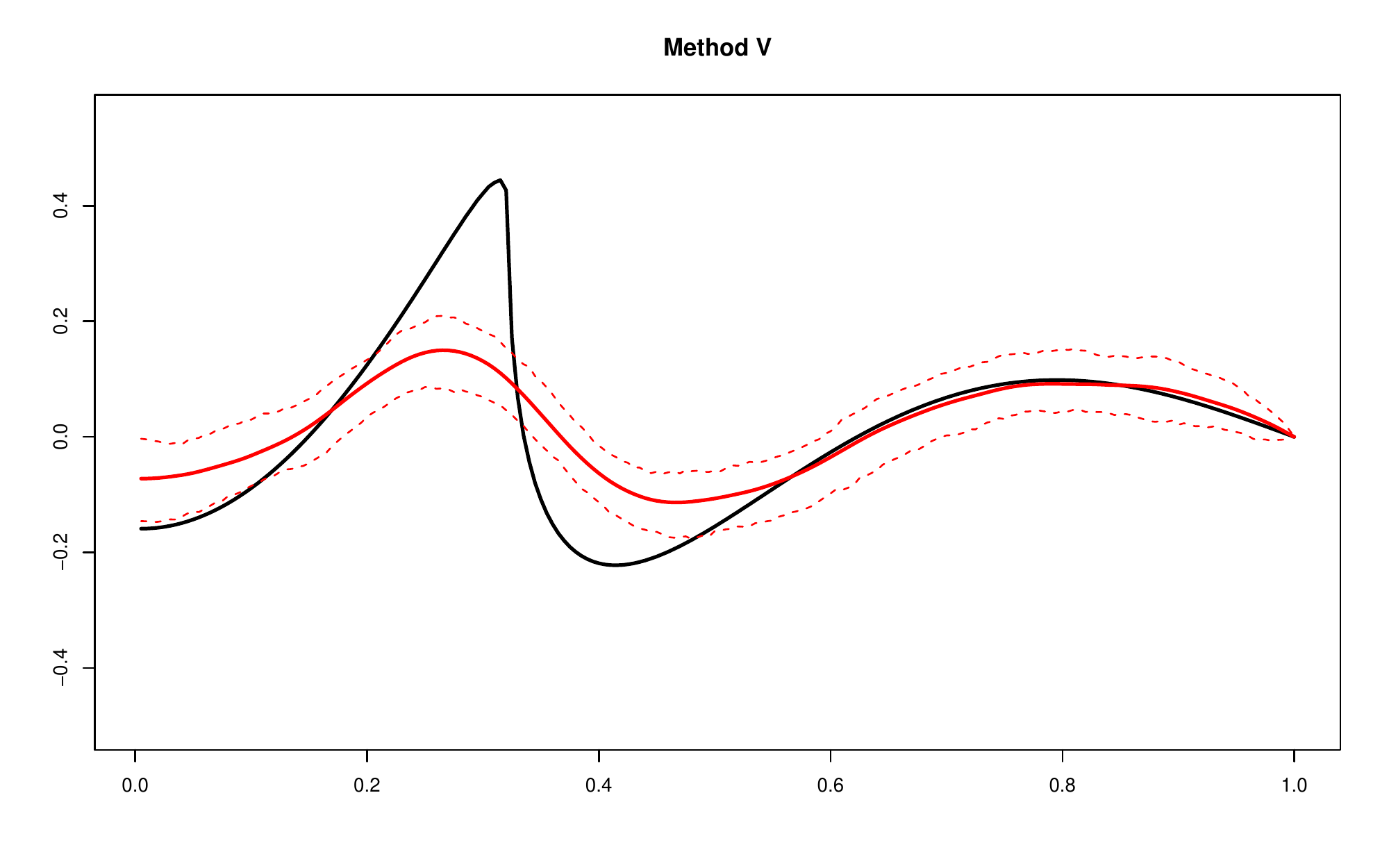}
\includegraphics[scale=0.3]{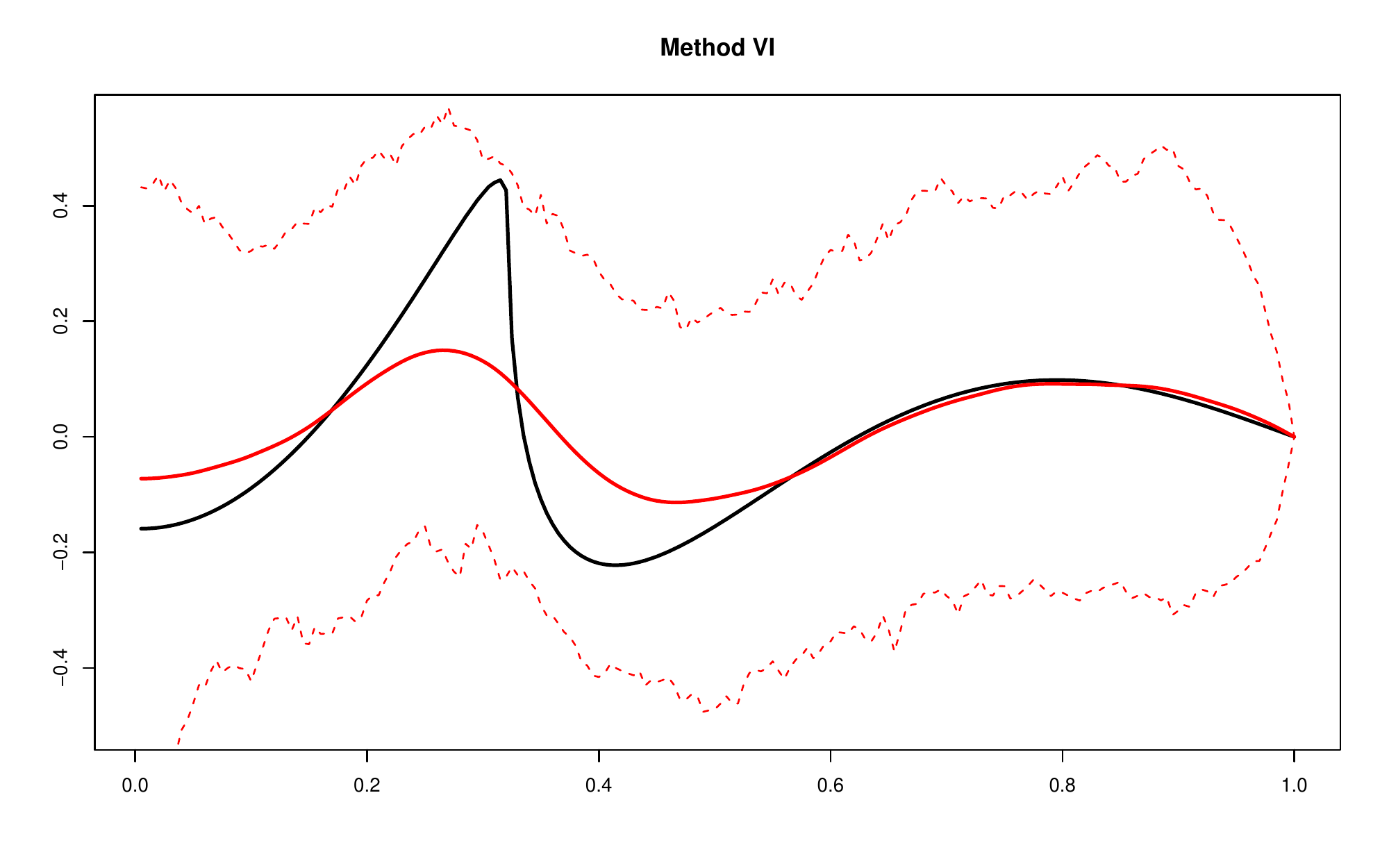}
\caption{Global posterior mean (solid red curve) and $95\%$ pointwise credible bands (dashed red curves) for each of the methods I--VI.}
\label{fig: dis2}
\end{figure}

An important observation to make is that the methods that 
achieve the same optimal performance as non-distributed 
methods, all use information about the regularity $\beta$ 
of the unkown signal, mostly through the setting of tuning parameters in 
the priors. In that sense, they are non-adaptive.
They serve as useful results that indicate what is possible 
in principle if we have certain oracle knowledge about the 
truth we are trying to learn. To understand what realistic 
procedures can achieve 
this has to be combined with insight into what can be 
learned about this oracle knowledge from the data. 
In the next section we address this issue in the context
of our distributed signal-in-white-noise model.

\section{Results for adaptive procedures}
\label{sec: adap}

In the non-distributed case it is well known that there 
exist  adaptive methods that  achieve the same 
optimal performance as non-adaptive procedures,  without 
using knowledge of the regularity $\beta$
of the unkown signal. 
These methods somehow succeed in correctly trading off
bias, variance (and spread in Bayesian methods) in 
a purely data-driven manner.  For several such result in 
the context of the signal-in-white-noise model, see, for instance, 
\cite{gine:nickl:2016} and the references therein.
For distributed methods the issue of adaptation appears
to be a lot more subtle. In this paper we only have a first, 
negative result on adaptive properties of distributed methods. 

So now we do not  assume that we know  the true 
regularity $\beta$ of the unknown signal. 
As before we employ the prior $\Pi(\cdot\given\alpha)$  
in the local machines. 
To tune the regularity parameter $\alpha$ of the prior 
we  consider a distributed version of maximum marginal 
likelihood, as proposed by \cite{deisenroth:2015}.
The usual, non-distributed version of that method 
would use the maximizer of the  map
\[
\alpha \mapsto \log \int \Big(\prod_{j=1}^m p(\mathbf Y^j \given \theta)\Big)\,\Pi(d\theta \given \alpha)
\]
as tuning parameter. 
Maximizing this function however requires 
having all data available in a central machine. In the distributed setting, 
\cite{deisenroth:2015} argue that this map is well approximated by the 
map 
\[
\alpha \mapsto \sum_{j=1}^m \log \Big(\int  p(\mathbf Y^j \given \theta)\,\Pi(d\theta \given \alpha)\Big).
\]
Now every term in the sum just depends on one of the local machines and 
this function can be maximized on the central machine by repeatedly
asking the local machines for function evaluations and gradients 
of the local log-marginal likelihoods
\[
\log \int  p(\mathbf Y^j \given \theta)\,\Pi(d\theta \given \alpha).
\] 
The resulting estimator is denoted by $\hat\alpha$, i.e.\
\[
\hat\alpha = \argmax_{\alpha \in [0, \log n]}
\sum_{j=1}^m \log \Big(\int  p(\mathbf Y^j \given \theta)\,\Pi(d\theta \given \alpha)\Big).
\]
(We maximize over a compact interval to ensure that the maximizer exists.)

It turns out  that in the distributed setting, the local machines 
are in general not able to learn enough about the true signal regularity $\beta$. 
The following lemma asserts that there exist ``difficult'' signals for which the  
estimator $\hat\alpha$ overestimates the 
regularity. 

\bigskip

\begin{lemma}\label{lem: counter MMLE}
For $\beta, M > 0$, consider a signal $\theta_0 \in \ell^2$ such that 
\begin{equation}\label{eq: counter}
\theta_{0,i}^2=\begin{cases}
M^2 i^{-1-2\beta}& \text{if $i\geq (n/(\sigma^2\sqrt{m}))^{1/(1+2\beta)}$},\\
0& \text{else}.
\end{cases}
\end{equation}
Then $\theta_0 \in H^\beta(M)$ and if $M$ is small enough, then
\begin{align}
\PP_{\theta_0}(\hat\a\ge\beta+1/2)\to 1\label{eq:help}
\end{align}
if $n/m \to \infty$ and $m \to \infty$.
\end{lemma}

\begin{proof}
The proof is given in Section \ref{sec: counter MMLE}.
\end{proof}

\bigskip

In view of Lemma \ref{lem: counter MMLE} it is perhaps not surprising that 
if the approximated maximum marginal likelihood estimator $\hat\alpha$ 
is used  to tune the local prior that is used in every machine, sub-optimal performance 
is obtained for certain truths. Intuitively this is because due to the smaller 
signal-to-noise ratio, or ``sample size'' in the local machines, certain truths may appear 
more regular than they really are. It turns out that using the estimator $\hat\alpha$ 
in combination with any of the methods considered in the preceding section indeed 
leads to sub-optimal rates and bad coverage probabilities for certain truths.
As an illustration we present a rigorous statement for the method of Section \ref{sec: was}, 
but similar results can be derived for the others methods as well.

So suppose that in every local problem the prior $\Pi(\cdot\given \alpha)$ 
is used,   the corresponding generalized posterior 
$\tilde \Pi^j(\cdot \given \mathbf Y^j)$ is computed locally (which involves 
raising the local likelihood to the power $m$), and then the tuning parameter 
$\alpha$ is substituted by the estimator $\hat\alpha$ defined above.
In the central machine,  the global ``posterior'' $\mathbf \Pi_{VII}(\cdot\given \mathbf Y)$
is constructed as the 2-Wasserstein barycenter
of the local ``posterior'' measures $\tilde \Pi^1(\cdot \given \hat\alpha, \mathbf Y^1), 
\ldots, \tilde \Pi^m(\cdot \given \hat\alpha, \mathbf Y^m)$.

\bigskip

\begin{theorem}\label{thm: counter MMLEB:Wasser}
For $\beta, M > 0$  and $\theta_0$ 
as in Lemma \ref{lem: counter MMLE} we have, for some $c > 0$,  
\begin{align*}
\EE_{\theta_0}\mathbf\Pi_{VII}(\theta:\, \|\theta-\theta_0\|_2\leq c(n/\sqrt m)^{-\frac{\beta}{1+2\beta}}|\mathbf{Y} )\rightarrow 0
\end{align*}
as $m \to \infty$ and $n/m \to \infty$. Furthermore, for all $L>0$ it holds that 
\[
\PP_{\theta_0}\big(\theta_0\in \hat{C}(L) \big)\rightarrow 0.
\]

\end{theorem}

\begin{proof}
See Section \ref{sec: counter_MMLEB:Wasser}.
\end{proof}

\bigskip

A simulation illustrating the theoretical result of theorem is 
given in Figure \ref{fig: dis}. The left panel visualizes the 
``posterior'' generated by method VII, in the same distributed setting, 
and using the same simulated data as 
considered in Section \ref{sec: summary}.

So when combined with  a  data-driven 
tuning method like the distributed version of maximum marginal likelihood 
considered here, 
even the distributed methods that perform well in the non-adaptive
setting loose their favourable properties. None of the methods yields
a procedure that automatically adapts to regularity and achieves the optimal non-distributed 
rate. This does not imply of course that such an adaptive method does not exist. 
We expect however that the matter is delicate and that fundamental limitations 
exist. 

The issue appears to be similar to that of the existence of 
adaptive confidence sets. To achieve adaptation in our distributed 
setting the local machines must be able to learn the ``global'' regularity 
of the signal from the limited local data that they have available. 
Analogous to the adaptive confidence problem we expect that this is in general 
only possible under additional assumptions on the true signal, 
like the self-similarity or polished tail conditions proposed for instance 
in \cite{ginŽ2010}, \cite{bull2012}, \cite{szabo:etal:2015}—, \cite{nickl2016}, 
\cite{belitser2017}. Making these admittedly somewhat loose claims
mathematically precise will take considerably more  effort, but seems 
an important and interesting direction for future work.

\appendix

\section{Proofs for Section \ref{sec: nonadapt}}
\label{sec: proofs}

\subsection{Proof of Theorem \ref{thm: naive_avg}}\label{sec: naive_avg}

By completing the square we see  that under the local posterior 
$\Pi^j(\cdot \given  \mathbf Y^j)$
the coefficients $\theta_i$ are  independent and Gaussian, 
with mean $\hat\theta^j_i$ and variance $s^2_i$ given by 
\[
\hat\theta^j_i = \frac{n}{n+\sigma^2mi^{1+2\beta}}Y^j_i, \qquad
s^2_i = \frac {\sigma^2m}{n+\sigma^2mi^{1+2\beta}}.
\]
Hence the global ``posterior'' $\mathbf\Pi_I(\cdot\given \mathbf Y)$ is  Gaussian 
as well, and under that measure the coefficients $\theta_i$ are independent and have
 mean $\hat\theta_i$ and 
variance $t^2_i$ given by 
\[
\hat\theta_i = \frac{1}{m}\sum_{j=1}^m \hat\theta^j_i, \qquad
t^2_i = \frac{s^2_i}{m}.
\]
For the global posterior mean we have, for every $\theta_0 \in \ell^2$,  
\[
\EE_{\theta_0}\, \hat\theta_i -\theta_ {0,i} = 
\frac{-\sigma^2mi^{1+2\beta}}{n+\sigma^2mi^{1+2\beta}}\theta_{0,i}, 
\qquad 
\Var_{\theta_0}\, \hat\theta_i = \frac{\sigma^2n}{(n+\sigma^2mi^{1+2\beta})^2},
\]
and hence, 
\begin{align*}
\EE_{\theta_0} \|\hat\theta-\theta_0\|^2_2  = \sum\frac{\sigma^4m^2i^{2+4\beta}}{(n+\sigma^2mi^{1+2\beta})^2}\theta_{0,i}^2 + 
\sum \frac{\sigma^2n}{(n+\sigma^2mi^{1+2\beta})^2}.
\end{align*}
By   Lemma A.1 of  \cite{szabo:etal:2013} the second, variance term is of the order 
\[
 m^{-1/(1+2\beta))}n^{-2\beta/(1+2\beta)},
\]
as $n/m \to \infty$. 
For  $\theta^2_{0,i}=Mi^{-1-2\beta}$, 
by the same lemma, the first, squared bias term is proportional to 
$(n/m)^{-2\beta/(1+2\beta)}$.
For the global spread, we have
\begin{align}
\sum t^2_i = \sum \frac {\sigma^2}{n+\sigma^2mi^{1+2\beta}} \asymp  m^{-1/(1+2\beta))}n^{-2\beta/(1+2\beta)}.\label{eq: naive_avg_postspread}
\end{align}

By the triangle inequality we have
\begin{align*}
& \mathbf \Pi_{I}(\|\theta-\theta_0\|_2\leq c m^{\frac{\beta}{1+2\beta}}n^{-\frac{\beta}{1+2\beta}}|\mathbf{Y})\\
&\leq \mathbf\Pi_{I}(\theta:\,\| \EE_{\theta_0}\hat\theta-\theta_0\|_2- c m^{\frac{\beta}{1+2\beta}}n^{-\frac{\beta}{1+2\beta}}-\|\hat\theta-\EE_{\theta_0}\hat\theta\|_2\leq\| \theta-\hat\theta\|_2 |\mathbf{Y}).
\end{align*}
It follows from the bounds on the variance and squared bias of the posterior mean 
that for $\theta_0$ as chosen above, the quantity 
\[
\| \EE_{\theta_0}\hat\theta-\theta_0\|_2- c m^{\frac{\beta}{1+2\beta}}n^{-\frac{\beta}{1+2\beta}}-
\|\hat\theta-\EE_{\theta_0}\hat\theta\|_2
\]
appearing in the posterior probability 
 is with $\PP_{\theta_0}$-probability tending to one bounded from below by $c m^{\frac{\beta}{1+2\beta}}n^{-\frac{\beta}{1+2\beta}}$ for $c>0$ small enough. 
 Then by the upper bound for the posterior spread and Chebyshev's inequality we 
 obtain the first statement of the theorem.

For the coverage we note that the radius $r_{\gamma}$ of the credible set is a multiple of $m^{-1/(2+4\beta)}n^{-\beta/(1+2\beta)}$, which follows from the Gaussianity of the posterior and $\eqref{eq: naive_avg_postspread}$. Then by similar computations as above we get that
for the same truth $\theta_0$, 
\begin{align*}
\PP_{\theta_0}(\theta_0\in \hat{C}(L))&=\PP_{\theta_0}(\|\hat\theta-\theta_0\|_2\leq L  r_\gamma)\\
&\leq \PP_{\theta_0}\big(\|\hat\theta-\EE_{\theta_0}\hat\theta\|_2\geq \|\EE_{\theta_0}\hat\theta-\theta_0\|_2-Lr_\gamma\big)\\
&\leq  \PP_{\theta_0}\big(\|\hat\theta-\EE_{\theta_0}\hat\theta\|_2\geq c m^{\frac{\beta}{1+2\beta}}n^{-\frac{\beta}{1+2\beta}} \big)\\
&\lesssim  m^{\frac{-2\beta}{1+2\beta}}n^{\frac{2\beta}{1+2\beta}} \EE_{\theta_0}\|\hat\theta-\EE_{\theta_0}\hat\theta\|_2^2
\lesssim  m^{-1}\rightarrow 0.
\end{align*}
This completes the proof of the theorem.

\subsection{Proof of Theorem \ref{thm: likelihood_average}}\label{sec: likelihood_average}

Raising the  local likelihood \eqref{eq: lik} to the power $m$ makes it proportional to 
\[
\prod e^{-\frac12 \frac{n(Y^j_i-\theta_i)^2}{\sigma^2}}, 
\]
which is the likelihood for the case $m=1$. 
It follows that under the generalized local posterior $\tilde \Pi^j(\cdot \given  \mathbf Y^j)$
the coefficients $\theta_i$ are  independent and Gaussian, 
with mean $\hat\theta^j_i$ and variance $s^2_i$ given by 
\[
\hat\theta^j_i = \frac{n}{n+\sigma^2i^{1+2\beta}}Y^j_i, \qquad
s^2_i = \frac {\sigma^2}{n+\sigma^2i^{1+2\beta}}.
\]
Hence the global ``posterior'' $\mathbf\Pi_{II}(\cdot\given  \mathbf Y)$ is  again Gaussian, and under this global measure the coefficients $\theta_i$ are independent and have
 mean $\hat\theta_i$ and 
variance $t^2_i$ given by 
\[
\hat\theta_i = \frac{1}{m}\sum_{j=1}^m \hat\theta^j_i, \qquad
t^2_i = \frac{s^2_i}{m}.
\]
For the global posterior mean we have in this case, for every $\theta_0 \in \ell^2$,  
\[
\EE_{\theta_0}\, \hat\theta_i -\theta_ {0,i} = 
\frac{-\sigma^2i^{1+2\beta}}{n+\sigma^2i^{1+2\beta}}\theta_{0,i}, 
\qquad 
\Var_{\theta_0}\, \hat\theta_i = \frac{\sigma^2n}{(n+\sigma^2i^{1+2\beta})^2},
\]
and hence, 
\begin{align*}
\EE_{\theta_0} \|\hat\theta-\theta_0\|^2_2  = \sum\frac{\sigma^4i^{2+4\beta}}{(n+\sigma^2i^{1+2\beta})^2}\theta_{0,i}^2 + 
\sum \frac{\sigma^2n}{(n+\sigma^2i^{1+2\beta})^2}.
\end{align*}
For all $\theta_0 \in H^\beta(M)$, 
the squared bias term is bounded by 
\[
M^2 \sum \frac{\sigma^4i^{1+2\beta}}{(n+\sigma^2i^{1+2\beta})^2}
\lesssim M^2 n^{-2\beta/(1+2\beta)}
\]
for large $n$, and the variance term behaves like a constant times 
$n^{-2\beta/(1+2\beta)}$ as well. 
The global spread $\sum t^2_i$ is of the order 
$m^{-1}n^{-2\beta/(1+2\beta)}$ for large $n$.

For $M_n \to \infty$ and $\theta_0 \in H^\beta(M)$ we now have, by the triangle inequality, 
\begin{align*}
&\mathbf \Pi_{II}(\theta:\, \|\theta-\theta_0\|_2\geq M_n n^{-\beta/(1+2\beta)}\given \textbf{Y})\\
&\leq \mathbf \Pi_{II}(\theta:\, \|\theta-\hat\theta\|_2\geq M_n n^{-\beta/(1+2\beta)}-
\|\hat\theta-\EE_{\theta_0}\hat\theta\|_2-\|\theta_0-\EE_{\theta_0}\hat\theta\|_2\given\textbf{Y}).
\end{align*}
By the bounds on the bias and the variance of the posterior mean derived above the quantity on the right of the inequality in the last posterior probability is bounded from below by $(M_n/2) n^{-\beta/(1+2\beta)}$ 
with $P_{\theta_0}$-probability tending to one as $n, m \to \infty$, uniformly 
in $\theta_0 \in H^\beta(M)$.
By Chebychev's inequality, and the bound on the posterior spread, we conclude that 
the first statement of the theorem holds.

For the second statement we first note that 
by Chebychev's inequality and by the upper bound on the posterior spread the radius $r_\gamma$
of the credible set is for large $n$ bounded by  $Cm^{-1/2}n^{-\beta/(1+2\beta)}$ for some $C > 0$.
Hence, since the posterior mean is Gaussian, Anderson's inequality implies that 
\begin{align*}
\PP_{\theta_0}\big( \theta_0 \in \hat{C}(L)\big)&\leq \PP_{\theta_0}
\big(\|\hat\theta-\theta_0\|_2\leq CL m^{-1/2}n^{-\beta/(1+2\beta)} \big)\\
&\leq  \PP_{\theta_0}\big(\|\hat\theta-\EE_{\theta_0}\hat\theta\|_2\leq C L m^{-1/2}n^{-\beta/(1+2\beta)} \big).
\end{align*}
By Chebychev's inequality, 
\[
\PP_{\theta_0}\big(\|\hat\theta-\EE_{\theta_0}\hat\theta\|^2_2\leq 
\sum \sigma^2_i - a\sqrt{2\sum\sigma^4_i} \big) \le \frac 1{a^2}
\]
for all $a > 0$, where $\sigma^2_i = \Var_{\theta_0}\hat\theta_{i}$.
Above we saw that $\sum \sigma^2_i \asymp n^{-2\beta/(1+2\beta)}$. 
Similarly, it is  easily seen that $\sum \sigma^4_i \asymp n^{(-1-4\beta)/(1+2\beta)}$.
Hence by taking $a = n^{(1/4) /(1+2\beta)}$, for instance, we see that 
for $c > 0$ small enough, 
\[
\PP_{\theta_0}\big(\|\hat\theta-\EE_{\theta_0}\hat\theta\|_2\leq c n^{-\beta/(1+2\beta)} \big) \to 0
\]
as $n \to \infty$. But then also 
\begin{align*}
\PP_{\theta_0}\big(\|\hat\theta-\EE_{\theta_0}\hat\theta\|_2\leq C L m^{-1/2}n^{-\beta/(1+2\beta)} \big)\rightarrow 0
\end{align*}
as $m,n \to \infty$.

\subsection{Proof of Theorem \ref{thm: prior_average}}\label{sec: prior_average}

In this case the $j$th  local posterior is a product of Gaussians with means 
and variances given by 
\[
\hat\theta^j_i = \frac{n}{n+\sigma^2m\tau^{-1}i^{1+2\alpha}}Y^j_i, \qquad
s^2_i = \frac {\sigma^2m}{n+\sigma^2m\tau^{-1}i^{1+2\alpha}}.
\]
As before the global ``posterior''  is  Gaussian 
as well, and under that measure the coefficients $\theta_i$ are independent and have
 mean $\hat\theta_i$ and 
variance $t^2_i$ given by 
\[
\hat\theta_i = \frac{1}{m}\sum_{j=1}^m \hat\theta^j_i, \qquad
t^2_i = \frac{s^2_i}{m}.
\]
For the global posterior mean we have, for every $\theta_0 \in \ell^2$,  
\[
\EE_{\theta_0}\, \hat\theta_i -\theta_ {0,i} = 
\frac{-\sigma^2m\tau^{-1}i^{1+2\alpha}}{n+\sigma^2m\tau^{-1}i^{1+2\alpha}}\theta_{0,i}, 
\qquad 
\Var_{\theta_0}\, \hat\theta_i = \frac{\sigma^2n}{(n+\sigma^2m\tau^{-1}i^{1+2\alpha})^2},
\]
and hence 
\begin{align*}
\EE_{\theta_0} \|\hat\theta-\theta_0\|^2_2  = \sum\frac{\sigma^4m^2\tau^{-2}i^{2+4\alpha}}{(n+\sigma^2m\tau^{-1}i^{1+2\alpha})^2}\theta_{0,i}^2 + 
\sum \frac{\sigma^2n}{(n+\sigma^2m\tau^{-1}i^{1+2\alpha})^2}.
\end{align*}
By considering Riemann sums as in Lemma A.1 of \cite{szabo:etal:2013} we see that for $\beta < 1+2\alpha$ and 
uniformly for $\theta_0 \in H^\beta(M)$, the squared bias term is bounded 
by a constant times 
\begin{align*}
 M^2(\tau n/m)^{-2\beta/(1+2\alpha)}.
\end{align*}
Similarly, the variance term and the posterior spread $\sum t^2_i$ both behave like 
a constant times 
\[
(\tau/m)^{1/(1+2\alpha)} n^{-2\alpha/(1+2\alpha)}
\]
as $n \to \infty$. The choice $\tau = m n^{2(\a-\beta)/(1+2\beta)}$ 
balances these quantities, so that  all three are of the order $n^{-2\beta/(1+2\beta)}$.

By exactly the same reasoning as in Section \ref{sec: likelihood_average}, the fact that the squared bias bound 
and the variance and spread are of the same order implies the first statement of the theorem.
For the coverage statement we first note that the squared credible set radius $r^2_\gamma$ 
is the $1-\gamma$ quantile of the distribution of $\sum t^2_i Z^2_i$, with $t^2_i$ as above and 
$Z_i$ independent standard normals. This distribution has mean 
$\sum t^2_i \asymp n^{-2\beta/(1+2\beta)}$ 
and variance 
\[
2\sum t^4_i \asymp \frac 1n n^{-2\beta/(1+2\beta)},
\] 
as can be seen by considering Riemann sums again. As the standard deviation is of smaller order 
than the mean, it follows from Chebychev's inequality that $r_\gamma \ge c n^{-\beta/(1+2\beta)}$
for some $c > 0$. For the coverage probability we then have 
\begin{align*}
\PP_{\theta_0}\Big(\theta_0\not \in \hat{C}(L) \Big) \le 
\PP_{\theta_0}\Big(\|\hat\theta - \theta_0\|_2 \ge cLn^{-\beta/(1+2\beta)}\Big)
\le \frac{n^{2\beta/(1+2\beta)}}{c^2L^2}\EE_{\theta_0} \|\hat\theta - \theta_0\|^2_2.
\end{align*}
By the bounds on the bias and variance of the posterior mean the right-hand side is 
smaller than $\gamma$ for $L$ large enough, uniformly for $\theta_0 \in H^\beta(L)$.

\subsection{Proof of Theorem \ref{thm: likelihood_Wasserstein}}\label{sec: likelihood_Wasserstein}

As we saw in Section \ref{sec: likelihood_average}, the $j$th local generalized  
posterior is a product of Gaussians with  means $\hat\theta^j_i$ and variances $s^2_i$ given by 
\[
\hat\theta^j_i = \frac{n}{n+\sigma^2i^{1+2\beta}}Y^j_i, \qquad
s^2_i = \frac {\sigma^2}{n+\sigma^2i^{1+2\beta}}.
\]
In other words, the $j$th local measure is a Gaussian measure on $\ell^2$ 
with mean $\hat\theta^j = (\hat\theta^j_i)_i$ and (diagonal) covariance operator $R: \ell^2 \to \ell^2$
given by $(R x)_i = s^2_i x_i$, which is the same for every local machine. 
The Wasserstein barycenter of a finite collection of Gaussian measures 
is a Gaussian measure again (e.g.\ \cite{agueh2011}). 
By Theorem 3.5 of \cite{gelbrich1990}  the squared $2$-Wasserstein distance
between the $j$th local measure and a Gaussian measure on $\ell^2$ 
with mean $\mu$ and covariance operator $K$ is given by 
\[
\|\hat\theta^j -\mu\|_2^2 +  \tr(R) + \tr(K) -2\tr\sqrt{R^{1/2}KR^{1/2}}.
\]
It follows that the barycenter $\mathbf \Pi_{IV}(\cdot\given \mathbf Y)$
of the local generalized posteriors is the  Gaussian measure on $\ell^2$
with mean $\hat\theta$ equal to the average of the local means $\hat\theta^j$ 
and covariance operator equal to $R$. In other words, the global ``posterior'' 
is a product of Gaussians with means and variances given by 
\[
\hat\theta_i = \frac{1}{m}\sum_{j=1}^m \hat\theta^j_i, \qquad
t^2_i = {s^2_i}.
\]
So the global posterior mean is the same as in Section \ref{sec: likelihood_average}
and the posterior spread $\sum t^2_i$ is a factor $m$ larger. It then 
follows from the considerations in Section \ref{sec: likelihood_average}
that the squared bias of the global posterior mean is bounded by a constant times
$M^2n^{-2\beta/(1+2\beta)}$, uniformly for $\theta_0 \in H^\beta(M)$. 
Moreover, the variance term $\sum s^2_i$ and the posterior spread $\sum t^2_i$ 
behave like a multiple of $n^{-2\beta/(1+2\beta)}$ as well. 
As was explained in Section \ref{sec: prior_average}, this leads to the 
statement of the theorem.

\subsection{Proof of Theorem \ref{thm: gpoe}}\label{sec: gpoe}

The proof of the first statement is the same as in Section \ref{sec: naive_avg}, 
since the mean of the global ``posterior'' is the same 
as for the naive averaging method. 

For the second statement, we observe that for $\theta_0 \in H^\beta(M)$, 
the squared bias term for the posterior mean satisfies
\[
\sum\frac{\sigma^4m^2i^{2+4\beta}}{(n+\sigma^2mi^{1+2\beta})^2}\theta_{0,i}^2  \le
M^2 \sum\frac{\sigma^4m^2i^{1+2\beta}}{(n+\sigma^2mi^{1+2\beta})^2} \lle M^2 (n/m)^{-2\beta/(1+2\beta)}.
\]
for $n/m \to \infty$. 
As was shown in Section \ref{sec: naive_avg} the variance of the posterior mean 
behaves as  $m^{-1/(1+2\beta)}n^{-2\beta/(1+2\beta)}$. Since the spread  $\sum t_i^2$
of the posterior is a factor $m$ larger than in Section \ref{sec: naive_avg}, 
it is of the same order $(n/m)^{-2\beta/(1+2\beta)}$ as the squared bias term. 
Since squared bias and spread are of the same order, the variance is of 
smaller order, and 
\[
\sqrt{\sum t^4_i} \asymp \Big(\frac nm\Big)^{\frac{-1/2}{1+2\beta}}\sum t^2_i
\]
is of lower order than $\sum t^2_i$, 
the coverage statement can be proved as in Section \ref{sec: prior_average}.

\section{Proofs for Section \ref{sec: adap}}
\label{sec: proofs2}

\subsection{Proof of Lemma \ref{lem: counter MMLE}}\label{sec: counter MMLE}

The estimator $\hat\alpha$ is the maximizer of the 
random map $\alpha \mapsto \sum_j \ell_j(\alpha)$, where
\[
\ell_j(\alpha) = \log \int  p(\mathbf Y^j \given \theta)\,\Pi(d\theta \given \alpha).
\]
The asymptotic behaviour of the local log-marginal likelihood $\ell_j$
has been studied in \cite{knapik:etal:2016}. Denote the derivative of $\ell_j$ with 
respect to $\alpha$ by $\dot\ell_j$ and let $k = n/(\sigma^2m)$ be the local ``sample size''. 
Moreover, for $l > 0$, define 
\[
\underline \alpha = \inf\{\alpha > 0: h_k(\alpha) > l\}\wedge \sqrt{\log k},
\]
where 
\[
h_k(\alpha) = \frac{1+2\alpha}{k^{1/(1+2\alpha)}\log k}\sum_i 
\frac{k^2i^{1+2\alpha}\theta^2_{0,i}\log i}{(k + i^{1+2\alpha})^2}.
\]
Note that the expectation $\EE_{\theta_0}\dot\ell_j(\alpha)$ does not 
depend on $j$. It is proved in Section 5.3 of \cite{knapik:etal:2016} that if $l$ is smaller 
than some universal threshold, then for every $j$
\[
\liminf_{k \to \infty} \inf_{\alpha \le \underline\alpha} 
\frac{1+2\alpha}{k^{1/(1+2\alpha)}\log k} \EE_{\theta_0}\dot\ell_j(\alpha) =\delta > 0,
\]
\[
\EE_{\theta_0} \sup_{\alpha \le \underline\alpha} 
\frac{1+2\alpha}{k^{1/(1+2\alpha)}\log k} |\dot\ell_j(\alpha) - \EE_{\theta_0}\dot\ell_j(\alpha)|
\lle e^{-C\sqrt{\log k}} 
\]
for constants $\delta, C > 0$. But then we also have 
\[
\liminf_{k \to \infty} \inf_{\alpha \le \underline\alpha} 
\frac{1+2\alpha}{k^{1/(1+2\alpha)}\log k} \EE_{\theta_0}\sum_j\dot\ell_j(\alpha) > m\delta
\]
and
\[
\EE_{\theta_0} \sup_{\alpha \le \underline\alpha} 
\frac{1+2\alpha}{k^{1/(1+2\alpha)}\log k} \Big|\sum_j\dot\ell_j(\alpha) - 
\EE_{\theta_0}\sum_j \dot\ell_j(\alpha)\Big|
\lle m e^{-C\sqrt{\log k}}. 
\]
By Markov's inequality, it follows that with probability at least $1-C_1\exp(-C_2\sqrt{k})$
the map $\alpha \mapsto \sum_j \ell_j(\alpha)$ is strictly increasing on 
the interval $[0, \underline\alpha]$. Hence, on that event we have $\hat\alpha \ge \underline\alpha$.

It remains to show that $\underline\alpha \ge \beta +1/2$. 
To that end it suffices to prove 
that $h_{k}(\alpha) \le l$ for all $\alpha \le \beta + 1/2$. 
To see this, suppose first  that $\a< \beta$. Define $N_{\beta}= (n/(\sigma^2\sqrt{m}))^{1/(1+2\beta)}$ 
and $M_{\a}=k^{1/(1+2\a)}$. 
By  definition of $\theta_0$ we then have
\begin{align*}
h_{k}(\a)&=\frac{M^2}{M_{\a}\log M_\a} \sum_{i=N_\b}^{\infty}\frac{k^2i^{2\a-2\beta}\log i}{(i^{1+2\a}+k)^2}\\
&\leq \frac{M^2}{M_{\a}\log M_\a} \sum_{i=N_\b}^{M_\a}i^{2\a-2\beta}\log i
+ \frac{M^2k^2}{M_{\a}\log M_\a} \sum_{i=M_\a}^{\infty}i^{-2-2\a-2\beta}\log i\\
&\leq M^2\frac{M_{\a}N_{\beta}^{2\a-2\beta}\log M_{\a}}{M_{\a}\log M_{\a}}
+M^2\frac{k^2M_{\a}^{-1-2\a-2\beta}\log M_{\a}}{M_{\a}\log M_{\a}}\\
& \lle M^2 
\end{align*}
for $n,m$ large enough. Hence, if $M$ is small enough, then $h_{k} \le l$
for $\alpha < \beta$. 
For $\beta \le \a\leq \beta+1/2$ we have 
\begin{align*}
h_{k}(\a)&\leq \frac{M^2k^2}{M_{\a}\log M_\a} \sum_{i=N_\b}^{\infty}i^{-2-2\a-2\beta}\log i\\
&\leq \frac{M^2k^2N_{\beta}^{-1-2\a-2\beta}\log N_{\beta}}{M_{\a}\log M_{\a}}\\
&=  M^2(n/\sigma^2)^{\frac{2\alpha}{1+2\a}-\frac{2\a}{1+2\beta}}m^{-2+\frac{1}{1+2\a}+\frac{1+2\a+2\b}{2(1+2\beta)}}\frac{\log N_\b}{\log M_{\a}} \lle m^{-\frac{2\a}{1+2\a}} \log m
\end{align*}
for $n/m$ large enough. 
Together, this shows that if both $n/m$ and $m$ are large enough, 
then indeed $h_{k}(\alpha) \le l$ for all $\alpha \le \beta + 1/2$.

\subsection{Proof of Theorem \ref{thm: counter MMLEB:Wasser}}\label{sec: counter_MMLEB:Wasser}
In view of the proof of Theorem \ref{thm: likelihood_Wasserstein} the $j$th local generalized
posterior  
is a product of Gaussians with  means $\hat\theta^j_i$ and variances $s_i^2$ given by 
\[
\hat\theta^j_i = \frac{n}{n+\sigma^2i^{1+2\hat\alpha}}Y^j_i, \qquad
s_i^2 = \frac {\sigma^2}{n+\sigma^2i^{1+2\hat\alpha}}.
\]
Using again that  the Wasserstein barycenter of a finite collection of Gaussian measures 
is a Gaussian measure in combination with the explicit expression for the 2-Wasserstein 
distance between Gaussians (see Section \ref{sec: likelihood_Wasserstein}) we see 
that the global ``posterior'' is a product of Gaussians with  means $\hat\theta_i$ and variances $t^2_i$ given by
\[
\hat\theta_i = \frac1m\sum_{j=1}^m \hat\theta^j_i, \qquad
t^2_i = s^2_i.
\]

The posterior mean can be written as $\hat\theta = \hat\theta(\hat\alpha)$, 
where $\hat\theta(\alpha)$ is the estimator with a fixed choice $\alpha$ for the hyperparameter, 
i.e.\ 
\[
\hat\theta_i(\alpha) = \frac1 m \sum_{j=1}^m \frac{n}{n+\sigma^2i^{1+2\alpha}}Y^j_i.
\]
For fixed $\alpha$ we also define the corresponding 
expectation $E(\alpha)  = \EE_{\theta_0} \hat\theta(\alpha)$.
Then by the triangle inequality, 
\[
\| \hat\theta - \theta_0\|_2 \ge \|E(\hat\alpha) - \theta_0\|_2 
- \|E(\hat\alpha) - \hat\theta(\hat\alpha)\|_2.
\]
We have the explicit expressions
\[
\| E(\alpha) - \theta_0\|^2_2  = \sum_i \frac{\sigma^4i^{2+4\alpha}\theta^2_i}
{(n+\sigma^2i^{1+2\alpha})^2}
\]
and
\[
\|E(\alpha) - \hat\theta(\alpha)\|^2_2 = \sum_i 
\frac{\sigma^2n}{(n+\sigma^2i^{1+2\alpha})^2}\Big(\frac1 {\sqrt m} \sum_{j=1}^m 
Z^j_i\Big)^2.
\]
Since the first expression is increasing in $\alpha$ and the 
second one is decreasing, we see that 
on the event $A = \{\hat\alpha\ge \beta+ 1/2\}$ it holds that 
\[
\| \hat\theta - \theta_0\|_2 \ge \sqrt{\sum_i \frac{\sigma^4\theta^2_{0,i} i^{4+4\beta}}{(n+\sigma^2i^{2+2\beta})^2}}
- \sqrt{\sum_i 
\frac{\sigma^2n}{(n+\sigma^2i^{2+2\beta})^2}\Big(\frac1 {\sqrt m} \sum_{j=1}^m 
Z^j_i\Big)^2}.
\]
By definition of $\theta_0$, the square of the first term on the right is bounded from below 
by 
\begin{align*}
M^2\sum_{i \ge (n/\sqrt m)^{1/(1+2\beta)}} \frac{ \sigma^4i^{3+2\beta}}{(n+\sigma^2i^{2+2\beta})^2}. 
\end{align*}
By comparing this to the corresponding Riemann sum we see that it is of 
the order $M^2(n/\sqrt m)^{-2\beta/(1+2\beta)}$.
The square of the second term can be written as 
\[
\sum_i 
\frac{\sigma^2n}{(n+\sigma^2 i^{2+2\beta})^2}U_i^2,
\]
with the $U_i$ independent and standard normal under $\PP_{\theta_0}$. 
By considering Riemann sums again, for instance, it is easily seen that 
the mean and variance of this sum behave
as $n^{-(1+2\beta)/(2+2\beta)}$ and 
$n^{-(3+4\beta)/(2+2\beta)}$, respectively. Hence the standard deviation is 
of smaller order than the mean for large $n$, so that by Chebychev's inequality the 
square of the second term is of stochastic order $n^{-(1+2\beta)/(2+2\beta)}$. 
Since this is of smaller order  than $(n/\sqrt m)^{-2\beta/(1+2\beta)}$, 
we conclude that for the global ``posterior'' mean we have, for some constant $c > 0$, 
\[
\PP_{\theta_0} (\| \hat\theta - \theta_0\|_2 \ge c (n/\sqrt m)^{-\beta/(1+2\beta)})
\to 1
\]
as $n/m \to \infty$ and $m \to \infty$. 
The spread $\sum t^2_i$ of the global posterior is on the event $A$ bounded by 
\[
\sum_i \frac1{n+i^{2+2\beta}}, 
\]
which is of the order $n^{-(1+2\beta)/(2+2\beta)} \ll (n/\sqrt m)^{-2\beta/(1+2\beta)}$ as well. 
The conclusions of the theorem now follow.

\newpage

\bibliographystyle{harry}
\bibliography{references}

\end{document}